\documentclass[11pt]{amsart}

\usepackage{graphics}

\makeatletter
\newcommand*\bigcdot{\mathpalette\bigcdot@{.5}}
\newcommand*\bigcdot@[2]{\mathbin{\vcenter{\hbox{\scalebox{#2}{$\m@th#1\bullet$}}}}}
\makeatother

\newcommand{\fcal}{{\mathcal F}}

\usepackage{amsmath}
\usepackage{amsthm}
\usepackage{amsfonts}
\usepackage{xcolor}
\usepackage[english]{babel}
\usepackage[margin=1.09in]{geometry}
\usepackage[colorlinks=true,
        linkcolor=blue]{hyperref}

        \definecolor{pink}{rgb}{1,0,1}

\usepackage[latin9]{inputenc}
\usepackage{amsmath}
\usepackage{amsfonts}
\usepackage{amssymb}
\usepackage{amsthm}

\newtheorem{theo}{Theorem}[section]
\newtheorem{prop}[theo]{Proposition}
\newtheorem{coro}[theo]{Corollary}
\newtheorem{lemm}[theo]{Lemma}

\theoremstyle{definition}
\newtheorem{def1}[theo]{Definition}
\theoremstyle{remark}
\newtheorem{rema}[theo]{Remark}

\newcommand{\nwc}{\newcommand}
\nwc{\eps}{\epsilon}
\nwc{\vareps}{\varepsilon}
\nwc{\Oph}{\operatorname{Op}_\hbar}
\nwc{\la}{\langle}
\nwc{\ra}{\rangle}

\nwc{\mf}{\mathbf} 
\nwc{\blds}{\boldsymbol} 
\nwc{\ml}{\mathcal} 

\nwc{\defeq}{\stackrel{\rm{def}}{=}}

\nwc{\cE}{\ml{E}}
\nwc{\cN}{\ml{N}}
\nwc{\cO}{\ml{O}}
\nwc{\cP}{\ml{P}}
\nwc{\cU}{\ml{U}}
\nwc{\cV}{\ml{V}}
\nwc{\cW}{\ml{W}}
\nwc{\tU}{\widetilde{U}}
\nwc{\IN}{\mathbb{N}}
\nwc{\IR}{\mathbb{R}}
\nwc{\IZ}{\mathbb{Z}}
\nwc{\IC}{\mathbb{C}}
\nwc{\IT}{\mathbb{T}}
\nwc{\tP}{\widetilde{P}}
\nwc{\tPi}{\widetilde{\Pi}}
\nwc{\tV}{\widetilde{V}}
\nwc{\supp}{\operatorname{supp}}
\nwc{\rest}{\restriction}

\renewcommand{\Re}{\operatorname{Re}}

\newcommand{\R}{{\mathbb R}}

\newcommand{\Z}{{\mathbb Z}}

\renewcommand{\d}{\partial}

\newcommand{\half}{{\frac{1}{2}}}

\renewcommand{\phi}{\varphi}

\newcommand{\lcal}{\mathcal{L}}

\newcommand{\scal}{\mathcal{S}}

\newcommand{\ep}{\varepsilon}

\title[One can hear the shape of ellipses of small eccentricity]
{One can hear the shape of ellipses of small eccentricity}

\author[Hezari]{Hamid Hezari }
\address{Department of Mathematics, UC Irvine, Irvine, CA 92617, USA} \email{hezari@math.uci.edu}

\author[Zelditch]{Steve Zelditch}
\address{Department of Mathematics, Northwestern University, Evanston, IL 60208, USA} \email{zelditch@math.northwestern.edu}


\date{\today}

\begin{document}

\begin{abstract} We show that if the eccentricity of an ellipse is sufficiently small then up to isometries it is spectrally unique among all smooth domains. We do not assume any symmetry, convexity, or closeness to the ellipse, on the class of domains. 

 In the course of the proof we also show that for nearly circular domains, the lengths of periodic orbits that are shorter than the perimeter of the domain must belong to the singular support of the wave trace. As a result we also obtain a Laplace spectral rigidity result for the class of axially symmetric nearly circular domains using a similar result of De Simoi, Kaloshin, and Wei concerning the length spectrum of such domains.

\end{abstract}

\maketitle

 \section{Introduction}

From the point of view of classical mechanics, elliptical billiards are very special because their billiard maps are completely integrable. In fact the Birkhoff conjecture asserts that ellipses are the only completely integrable strictly convex billiard tables. It is natural to expect this uniqueness property of ellipses to hold from the quantum mechanical point of view and ask for example whether the Laplace eigenvalues of ellipses with respect to Dirichlet or Neumann boundary conditions determine them uniquely. The only planar domains that are known to date to be determined by their spectrum among all smooth domains\footnote{In fact disks are spectrally unique among all Lipschitz domains by the isoperimetric inequality, because area and perimeter are spectral invariants of a Lipschitz domain by the heat trace asymptotic of Brown \cite{Br}.} are disks $D \subset \R^2$.  In this article we show that nearly circular ellipses are spectrally determined among all smooth domains. 
\begin{theo}\label{MAIN} There exists $\ep_0 >0$ such that any ellipse with eccentricity less than $\ep_0$ is uniquely determined by its Dirichlet (or Neumman) Laplace spectrum, among all smooth domains.  \end{theo}
Henceforth, we use the term `nearly circular ellipse' as short for ``eccentricity less than $\ep_0$.''
This inverse spectral result should be compared with the recent dynamical inverse results of Avila-De Simoi-Kaloshin  \cite{ADK} and Kaloshin-Sorrentino \cite{KS}. They prove a `local' version of the Birkhoff conjecture:  if a strictly convex (finitely smooth) planar domain is sufficiently close to an ellipse and is \textit{rationally integrable}, then it must be an ellipse. Rational integrabillity means that for every integer $q \geq 3$ there is a convex caustic of rotation number $\frac1q$ consisting of periodic orbits with $q$ reflections. In fact,  our proof is based in part  on this result. To be able to use it, we need to prove that the hypothesis is valid. We first need an important definition.
\begin{def1}\label{NC} Let $n\in \mathbb N$ and $\ep>0$. Let $D$ be the unit disk and $N_0$ be its outward unit normal. A simply connected planar domain $\Omega$ with smooth boundary will be called `$\ep$-nearly circular in $C^n$' if its boundary can be written as $\d \Omega = \d D + f(\theta) N_0$, with $\|f \|_{C^n(\d D)}=\cO_n(\ep)$. Here $\cO_n(\ep)$ means that  $\|f \|_{C^n(\d D)}$ is bounded by $A_n \ep$ for some $A_n$ that depends only on $n$. If we only use `nearly circular', it means that $\ep$ is sufficiently small. 
	\end{def1} 

The main advance in this article is contained in the following:

\begin{theo} \label{MAINTHEO2} If $\Omega$ is a bounded smooth plane domain which is isospectral to a nearly circular ellipse of eccentricity $\ep$, then $\Omega$ is $\ep$-nearly circular in $C^n$ for every $n \in \mathbb N$ (in particular it must be strictly convex) and $\Omega$ is rationally integrable. \end{theo}

The near circularity of $\Omega$ is proved in Proposition \ref{close}. The proof uses heat trace invariants
 to show that if a smooth domain $\Omega$ is isospectral to an ellipse $E$ with small eccentricity $\ep$, then $\Omega$ must be sufficiently close to $E$ in the $C^n$ norm for all $n$. In particular, $\Omega$ must itself be  almost circular.
 
 After this initial step,  the proof of rational integrability is based on a study of the  wave trace
 \begin{equation} \label{WTDEF} w_{\Omega}(t): =\rm{Tr} \cos t \sqrt{\Delta_\Omega}. \end{equation}
 It is well-known that $w_{\Omega}(t)$ is a tempered distribution on $\R$ and that the positive singularities of  $w_{\Omega}$ can only occur for $t \in \mathcal{L}(\Omega)$, the length spectrum (i.e. the closure of the set of lengths of closed billiard trajectories). Of particular importance here are the closed trajectories of type $\Gamma (1, q)$, i.e. with winding number $1$ and with $q$ bounces (reflections) off the
 boundary $\partial \Omega$. We denote the set of lengths of such closed
 trajectories by $\lcal_{1, q}(\Omega)$.  For each $q$, the contribution to $w_{\Omega}(t)$ of closed trajectories $\Gamma(1, q)$ is denoted by $\hat{\sigma}_{1,q}$.  In \cite[Proposition 6.11]{MM},  Marvizi-Melrose constructed microlocal  parametrices, also denoted by  $\hat{\sigma}_{1,q}$,
 for the microlocal contribution of trajectories in $\Gamma(1, q)$ and proved
 that the parametrix was valid for  $q \geq q_0(\Omega)$. By `valid' is meant that the wave trace is a sum of contributions from Lagrangian submanifolds $\Lambda_{q}$ corresponding to $q$-bounce orbits
 and $\hat{\sigma}_{1,q}$ is the contribution from those with winding number $p=1$ (see \cite[Section 6]{MM}).

To apply the results of \cite{ADK,KS} it is essential to have analogous results for $q \geq 3$ bounces.  One of the key results of this article is Theorem \ref{parametrix theorem}, which  shows that the Marvizi-Melrose parametrices are in fact valid for closed billiard trajectories in $\Gamma(1, q)$ with 
$q \geq 2$ for nearly circular domains in $C^8$. 

Theorem \ref{parametrix theorem} is applied in two  independent ways  to prove Theorem \ref{MAINTHEO2}. The first way is to combine it with a theorem of Soga \cite{So} for oscillatory integrals with degenerate phase functions to prove,
\begin{theo}\label{length spec = singsupp} Let $\Omega$ be a nearly circular domain in $C^8$.  Then, for all $q \geq 2$, one has $$\mathcal{L}_{1, q}(\Omega) \subset \text{SingSupp}\, w_\Omega(t).$$ In other words, for such domains, the wave trace is singular at the length of every $(1, q)$ periodic orbit. 
\end{theo}

Let us present the application of Theorem \ref{length spec = singsupp} to Theorem \ref{MAINTHEO2}. We let $\ell = |\partial \Omega|$ denote
the circumference. It is well-known to be a spectral invariant. 
 It is proved that, for  a nearly circular ellipse,   the singular support of $w_{\Omega}(t)$  contained in $(0, \ell)$  is a discrete set whose {\it gap sequence}  is monotonically decreasing. We refer to Lemma \ref{gaps decreasing}  for the definition and statement. On the other hand, if $\Omega$ is a nearly circular domain that is not rationally integrable then the gap sequence of the singular support must fluctuate. By Theorem \ref{length spec = singsupp} the
 lengths in $(0, \ell)$ with $q \geq 2$ are spectral invariants. If $\Omega$
 is isospectral to an ellipse of small eccentricity, then by Theorem \ref{length spec = singsupp}  its gap sequence is
 monotonically decreasing and therefore it is rationally integrable. 
 
 We then apply the  results of \cite{ADK} to show that $\Omega$ must be an ellipse. This step needs $\Omega$ to be a nearly circular in $C^n$ with $n=39$, which is provided to us, in fact for any $n$, by Theorem \ref{MAINTHEO2}. To conclude the proof, we use the easy result that if two ellipses are isospectral, then they must be isometric.

 \subsection{Second approach}
 The second application of Theorem \ref{parametrix theorem} uses the following: 
 
 \begin{prop} \label{MMSPINV} If $\Omega$ is a nearly circular domain in $C^8$,
 then for $q \geq 2$, 
 $\hat{\sigma}^{\Omega}_{1,q}(t)$ is a spectral invariant. Hence, if $\Omega$ is isospectral to an ellipse $E_\ep$ of small eccentricity $\ep$,
   then for all $q \geq 2$, we have $\hat{\sigma}^{\Omega}_{1,q}(t) = \hat{\sigma}^{E_\ep}_{1,q}(t)$. \end{prop}
  
  The statement is not obvious, because neither the winding number nor
  the bounce number are known to be spectral invariants. Moreover, if the length spectrum  $\mathcal{L}(\Omega)$ of $\Omega$ is multiple, i.e. if there
  exists more than one connected component in the set of closed billiard trajectories of some length $L$, then the contributions from the two components may cancel.  Theorem \ref{length spec = singsupp} shows that complete cancellation cannot occur, but Proposition \ref{MMSPINV} asserts more.
 
 Granted Proposition \ref{MMSPINV} the proof of Theorem \ref{MAINTHEO2} is
 rather simple: it is shown that the phase function of $\hat{\sigma}_{1,q}(t)$ 
 has exactly one critical value. But that forces it to be constant, and from
 that one sees that $\Omega$ must be rationally integrable.
 
\subsection{Application to spectral rigidity of $\Z_2$-symmetric nearly circular domains}
 In \cite{DKW} it is proved that the class $\mathcal S_\delta^8$ of axially symmetric planar domains that are $\delta$-nearly circular in $C^8$, $\delta$ sufficiently small, are length spectrally rigid within this class. Length spectral rigidity means that if $\Omega_t \in \mathcal S^8_\delta$, $0 \leq t \leq 1$, is a $C^1$ family along which the length spectrum is preserved, then $\Omega_t$ must be trivial, that is it consists of isometries of $\Omega_0$.  In fact the result of \cite{DKW} only uses the lengths of $(1, q)$-type periodic orbits.  Now equipped with our Theorem \ref{length spec = singsupp} , Lemma \ref{winding twice or more},  and Cor \ref{q can be heard}, we obtain the following result on the Laplace spectral rigidity of the class $\mathcal S_\delta^8$:
 
 \begin{theo} There exists $\delta >0$ universal such that the following statement holds. 
 	Let $\{\Omega_t\}_{t \in [0, 1]}$ be a $\Delta$ isospectral deformation of domains in the class $\mathcal  S_\delta^8$ of $\mathbb Z_2$-axially symmetric smooth domains that are $\delta$-close in $C^8$ to the unit disk. Then $\Omega_t$ must be a trivial deformation.  
 \end{theo} 

 \subsection{Prospects for more general ellipses} 
 It is natural to try to extend the results to more general ellipses or even more general convex domains. However,
 there exist many obstructions to any generalization of the proofs in this article for higher eccentricities. 
 
 An important step is the proof of Theorem \ref{parametrix theorem}, which shows that the parametrix for
 the $q$-bounce wave trace \eqref{sigma} of \cite[Proposition 6.11]{MM} is valid for $q \geq 3$ for a nearly circular
 domain. The proof of  Theorem \ref{parametrix theorem} is quite general but uses the `projectibility' of a certain
 portion of the $q$-bounce Lagrangian, namely the portion close to the diagonal consisting of orbits of winding number
approximately one. It is possible that this portion of the Lagrangian is projectible for more general ellipses.

 To explain the problem, we recall that the broken geodesic (billiard) flow
 induces a {\it billiard map} $\tilde \beta: B^*\partial \Omega \to B^*\partial \Omega$, where $B^* \partial \Omega$ is the unit `ball-bundle', which of
 course is an annulus in dimension $2$. See Section \ref{BILLIARDSECT} for background.  In the case of a convex domain, $\tilde\beta$ is a {\it twist map} of the annulus. This means that a `vertical' 
 $B^*_x \partial \Omega$ is mapped by $\tilde\beta$ to a horizontal curve
 $\tilde\beta(B^*_x \partial \Omega)$. Such a curve is of course a Lagrangian
 submanifold and may be parametrized by the differential of a function
 on the base $\partial \Omega$. However, $q$ bounce periodic  orbits are period
 $q$ orbits of $\tilde\beta$ and $\tilde\beta^q$ fails to be a twist map. In fact the
 image  $\tilde\beta^q(B^*_x \partial \Omega)$ folds over the base $q$ times.
 The essence of Theorem \ref{parametrix theorem} is to show that the piece of $\tilde\beta^q(B^*_x \partial \Omega)$ corresponding to the image of  small
 angles $\phi \in B^*_x \partial \Omega$, i.e. to billiard geodesic loops of winding
 number $1$, projects to $\partial \Omega$ without singularities near $x$.  
 Hence this piece may be parametrized by the differential of a function
 on $\partial \Omega$, namely, the $q$-bounce loop-length function for billiard loops
 at $x$ making $q$ bounces. 
 
 Note that the result of \cite{KS} extends \cite{ADK} to arbitrary ellipses. However the other steps of our arguments need closeness to a disk, so the proofs in this article do not extend without serious modifications to ellipses of arbitrary eccentricity. We also mention the works of Huang-Kaloshin-Sorrentino \cite{HKS} and a recent work of Koval \cite{Koval} that concern rational integrablity near the boundary of a nearly circular domain. In these results, it is only assumed that there exist caustics of rotations numbers $\frac{p}{q}$ where $\frac{p}{q} \leq \frac{1}{q_0}$ for a given (possibly large) $q_0$. There is hope that these methods can be extended to domains that are near ellipses of arbitrary eccentricity.

\subsection{Comparison to  works of Marvizi-Melrose and Amiran}
In \cite{MM},  Marvizi-Melrose used the parametrices to prove  that there exists a two-parameter
family of strictly convex domains which are spectrally determined
among domains satisfying a certain non-coincidence condition. The
domains are specified as solutions of extremal problems involving
 the so-called  Marvizi-Melrose invariants (see \cite{Sib99,Sib04} for the relation of these invariants to the marked length spectrum). They even show that the
curvature functions of the extremal domains are given by elliptic
integrals. But they do not conclude that the domains are ellipses.

In \cite{A,A2}, Amiran does state the  conclusions for the ellipse, but there appear to be serious gaps in the proof. The  present article over-laps \cite{A, A2} only in the proof of 
the {\it non-coincidence} condition. In \cite{A2} (see Corollary 7), the author shows that
the strong non-coincidence condition holds for an ellipse  whose
minor axis length exceeds $\frac{1}{4} \text{length}(\partial E)$. The proof does not appear to be complete and we give our own proof in the case of a nearly circular domain. 

 We briefly describe the approach of \cite{A,A2}.  In
\cite{A}, Amiran defines `caustics invariants' $L, J_1, G$ and states
(Theorem 9) that the extremals of $G$ among domains with fixed $L,
J_1$ are ellipses. The  non-coincidence condition (Theorem 10
of \cite{A}) is used to
 show that sufficiently many caustics invariants are $\Delta$-spectral
invariants.   The idea of the  proof is to show that only
curvature functions of ellipses solve the Euler-Lagrange equations
for $G$, a nonlinear second order equation for the radius of
curvature of the domain. We do not understand 
the proof given in \cite{A,A2}   that curvature
functions of ellipses solve the equation, or that they are the
only solutions. If indeed such ellipses are the only solutions
of the extremal problem,  then they would be
spectrally determined among domains whose curvature functions are
near that of the ellipse (Corollary 7).

\subsection{Previous positive results}  To our knowledge, the results of this paper give the first `universal inverse spectral result' for 
any class of  domains other than the circle. The result says that ellipses in a  specific family (`almost circular')  are determined by their spectra among all smooth domains without any further assumptions.
In fact, there do not even exist  prior `local spectral determination' results, which would say that an ellipse (or any other domain) is determined by its spectrum  among domains which lie in a sufficiently small $C^k$ neighborhood
of the ellipse. The only prior positive result specific to the inverse  Laplace spectral problem for ellipses is \cite{HeZeE} (see also \cite{PT3} for ellipsoids), which says that ellipses
are infinitesimally spectrally rigid among $C^{\infty}$ domains with the same left-right and up-down symmetries.  The progress in that article is to allow competing domains to be $C^{\infty}$ and not real-analytic.  To be precise, the rigidity result proved that any Dirichlet/Neumann isospectral deformation 
had to be `flat', i.e. all of its variational derivatives vanish. These results were generalized to all Robin boundary conditions in \cite{Vig1}.

The most general prior positive inverse results
were that of \cite{Ze}, where it is proved that a generic real analytic plane domain with one up-down symmetry is determined by its Dirichlet (or Neumann)
spectrum among other such domains, and that of \cite{DKW}, where a generic nearly circular domain with one reflection symmetry is shown to be Laplace spectrally rigid in the same class of domains.  In \cite{DKW}, the genericity was needed in drawing a conclusion on the Laplace spectral rigidity from the length spectral rigidity. The results of the present article, by comparison, do not make any symmetry assumptions and allow the competing domains to be general $C^{\infty}$ domains. There also exists a sequence of results of Popov-Topalov \cite{PT1, PT2,PT3} using the KAM structure of convex smooth plane domains to deduce  spectral rigidity results for Liouville billiards (including ellipses) with two commuting reflection symmetries, and for analytic domains that are sufficiently close to an ellipse and possess the two reflection symmetries of the ellipse. 

Prior inverse results for other classes of plane domains are surveyed in \cite{Z, DaHe, Z14}, with an emphasis on positive results. Negative
results, such as the construction of  isospectral polygonal domains of \cite{GWW}, are surveyed in \cite{Go}.

\subsection{Organization of the paper}
In Section 2, we show that any smooth domain isospectral to a nearly circular domain must be nearly circular as well. The main tool in proving this result is heat trace invariants.  Section 3 is dedicated to the existence of geodesics loops for nearly circular domains. The lengths of these loops (we call them loop functions) play a key role in our later analysis of the wave trace via the Marvizi-Melrose parametrix. We also find a useful variational formula for the loop functions. Section 4 entirely involves the length spectrum of nearly circular domains. The important gap structure of the length spectrum is proved in this section. The Marvizi-Melrose parametrix is proved in Section 5 for nearly circular domains for all $q \geq 2$. We provide an independent proof of this theorem using Green's second identity. In Section 6, we show that the part of the length spectrum that is less than the perimeter is contained in the singular support of the wave trace, hence in fact we obtain an equality in Poisson relation in the interval $(0, \ell)$. In Section 7, we prove our main result by showing that if a domain is isospectral to a nearly circular ellipse then it must be rationally integrable hence must be an ellipse by a result of Avila-De Simoi-Kaloshin.  At the end of  Section 7 we also provide an alternative proof of the rational integrability.

\section{Isospectrallity with a nearly circular ellipse implies closeness to the ellipse}
Let $E_\ep$ is an ellipse of eccentricity $\ep$. 
After a rescaling and a rigid motion, we can assume that $E_\ep$ is given by
$$ E_\ep= \left \{ (x, y) \in \R^2; \; x^2 + \frac{y^2}{1-\ep^2} \leq 1 \right \}.$$
Then assume $\Omega$ is a smooth domain with 
\begin{equation}\label{isospectral}
 \text{Spec}(\Omega) = \text{Spec} (E_\ep).
\end{equation}
Here, $ \text{Spec}$ means the spectrum of the euclidean Laplacian with respect to Dirichlet (or Neumman) boundary condition. We know from the heat trace invariants that $\Omega$ must be simply-connected with the same perimeter as $E_\ep$, which we shall call $\ell_\ep$ and use $s$ for the arclength parameter. We will also use
$$ \kappa(\Omega)(s) \quad \text{and} \quad \kappa(E_\ep)(s),$$
for the curvature functions of $\d \Omega$ and $\d E_\ep$ respectively.  Note that  $\kappa(\Omega)$ and $ \kappa(E_\ep)$ belong to the same space $C^\infty [0, \ell_\ep]$.
We now have the following lemma.
\begin{prop}\label{close} Suppose $\Omega$ and $E_\ep$ are isospectral. Then for all integers $n \geq 0$, we have
	$$ \| \kappa(\Omega)- 1 \| _{C^n[0, \ell_\ep]} = \cO_n (\ep).$$ In particular, for sufficiently small $\ep$, $\Omega$ is strictly convex. Here, $\cO_n(\ep)$ means that the involved constant depends only on $n$. 
\end{prop}
As a corollary we will obtain:
\begin{coro} \label{close 2} Suppose $\Omega$ and $E_\ep$ are isospectral. Then one can apply a rigid motion to $\Omega$ after which its boundary can be written as $\d \Omega = \d E_0 + f N_0$, with $\| f \|_{C^n(\d D)} = \cO_n(\ep)$ for all $ n\geq 0$. Here $N_0$ is the outward unit normal of the unit disk $D=E_0$.  In terms of Definition \ref{NC}, it means that $\Omega$ is $\ep$-nearly circular in $C^n$. 
	\end{coro}
\begin{proof}[Proof of Proposition \ref{close}]
Let $\Delta_\Omega$ be the positive Laplacian with Dirichlet (or Neumann) boundary condition on $\Omega$. We recall the well-known heat trace asymptotic
\begin{equation}\label{heat}
\text{Tr}( e^{-t \Delta_\Omega}) \sim t^{-1} \sum_{ n \geq 0} a_n t^n + b_nt^{n +\half}, \qquad t \to 0^+. 
 \end{equation}
 In \cite{Me}, the following structural property is proved for the heat invariants $b_n$:
 \begin{equation}\label{b_n}
 b_{n+1} = c_n \int_0^{\ell} \kappa_n^2 \, ds + \sum_ {\alpha \geq 0} d_\alpha \int_0^{\ell} \kappa^{\alpha_0} \kappa_1^{\alpha_1} \dots \kappa_{n-1}^{\alpha_{n-1}} \, ds,  \qquad n \geq 1,
 \end{equation}
 where $\ell$ is the length of the boundary,  $\kappa_m$ denotes $\frac{d^m \kappa}{ds^m}$, $\alpha=(\alpha_0, \dots, \alpha_{n-1})$ is a multi-index in $\Z^n$, $c_n \neq 0$ and $d_\alpha$ are universal constants. We also have $\sum_{j=0}^{k-1} (1+j)\alpha_j =2n+2$, which in particular implies that the sum is finite and the number of terms is bounded by a constant dependent only on $n$.  For us it is also important to know that 
 \begin{equation}\label{b1}
 b_1 = c_1 \int_0^\ell \kappa^2 \, ds, \quad c_1 \neq 0,
 \end{equation}
 and 
 \begin{equation}\label{b2}
 b_2 = c_2 \int_0^\ell \kappa_1^2 \, ds + c'_2 \int_0^\ell \kappa^4 \, ds, \quad c_2 \neq 0.
 \end{equation}
 Melrose used the trace invariants $b_n$ to prove a pre-compactness for the class of isospectral domains to a given domain $D$. More precisely he showed that for each smooth domain $\Omega_0$, and each $n \geq 0$, there is $A_n$ such that for all $\Omega$ isospectral to $\Omega_0$, we have
 \begin{equation}\label{compactness}
 || \kappa(\Omega)||_{C^n} \leq A_n.
 \end{equation} 
 Suppose now that $\Omega$ is isospectral to $E_\ep$. We would like to show that for all $n\geq 0$ 
 \begin{equation}\label{difference of curvatures}
 || \kappa(\Omega) - \kappa(E_\ep)||_{C^n} = \cO_n (\ep).
 \end{equation}
 
 Since $\kappa(E_0) =1$, we have $\kappa(E_\ep) =1 +\cO(\ep)$ and $\kappa_n(E_\ep) = \cO_n(\ep)$ for $n\geq 1$. Thus it is sufficient by the Sobolev embedding theorem to show that 
 \begin{align}\label{L2}
 \int_0^ {\ell_\ep} & | \kappa(\Omega) - \kappa(E_\ep)|^ 2 ds = \cO(\ep),
 \\
 \label{Hn} \int_0^{\ell_\ep} & \kappa_n^2(\Omega) \, ds = \cO_n(\ep), \quad  n \geq 1.
 \end{align}
 To see \eqref{L2}, we first use the invariant $b_1$ in \eqref{b1} and the fact $\ell_\ep = 2 \pi +\cO(\ep)$, to get 
 $$ \int_0^{\ell_\ep} \kappa^2(\Omega) \, ds =  \int_0^{\ell_\ep} \kappa^2(E_\ep) \, ds= 2 \pi+ \cO(\ep).$$
 This and the facts $\int_0^{\ell_\ep} \kappa(\Omega) ds=2\pi$ and $\kappa(E_\ep) =1 +\cO(\ep)$, imply that
 \begin{align*}
\int_0^ {\ell_\ep} | \kappa(\Omega) - \kappa(E_\ep)|^ 2\, ds = \int_0^ {\ell_\ep}  \kappa^2(\Omega) + \kappa^2(E_\ep) - 2 \kappa(\Omega)\kappa(E_\ep) \, ds = \cO(\ep).
  \end{align*}
  To prove \eqref{Hn}, we use \eqref{b_n} and argue by induction on $n \geq 1$. In the first step, we note that by the expression \eqref{b2} for the invariant $b_2$ we have:
  $$\int_0^{\ell_\ep} \kappa^2_1(\Omega) \, ds = \int_0^{\ell_\ep} \kappa^2_1(E_\ep) \, ds + \frac{c'_2}{c_2} \left ( \int_0^{\ell_\ep} \kappa^4(E_\ep) -  \kappa^4(\Omega)  \, ds \right )$$
  However, we can bound the last expression by $\cO(\ep)$ using \eqref{compactness}, Cauchy-Schwartz inequality, and \eqref{L2} as follows:
  \begin{align*}
 \int_0^{\ell_\ep} \kappa^2_1(\Omega) \, ds \leq \cO(\ep) + \frac{c'_2}{c_2} \left \| \kappa(E_\ep) -  \kappa(\Omega) \right \|_{L^2} \left \| \kappa^3(E_\ep) +  \kappa^2(E_\ep) \kappa(\Omega) +  \kappa(E_\ep) \kappa^2(\Omega)  + \kappa^3(\Omega) \right \|_{L^2}= \cO(\ep).\end{align*}
 Let us now assume $ \int_0^{\ell_\ep} \kappa^2_m(\Omega) \, ds = \cO(\ep)$ for $1 \leq m \leq n-1$. By \eqref{b_n}, we have
 $$\int_0^{\ell_\ep} \kappa_n^2(\Omega) = \int_0^{\ell_\ep} \kappa_n^2(E_\ep) + \frac{1}{c_n} \sum_{\alpha \geq 0} d_{\alpha} \int_0^{\ell_\ep}  \kappa^{\alpha_0}(E_\ep) \dots \kappa_{n-1}^{\alpha_{n-1}}(E_\ep) - \kappa^{\alpha_0}(\Omega)\dots \kappa_{n-1}^{\alpha_{n-1}} (\Omega).$$ To conclude the induction, we need to show that the right hand side of the above identity is $\cO_n(\ep)$. Obviously, by the induction hypothesis, and since for all $n\geq 1$, $\| \kappa_n(E_\ep)\|_{C^0} = \cO_n(\ep)$, all terms involving at least one derivative of the curvature are of size $\cO_n (\ep)$ (note that we still need the apriori bounds \eqref{compactness}). So it remains to estimate the part of the sum involving no derivatives, which is a sum of terms (up to multiplication by a constant) in the form:
 $$ \int_0^{\ell_\ep} \kappa^{\alpha_0}(E_\ep) - \kappa^{\alpha_0}(\Omega) \, ds. $$ Again, as in the first step, we factor   $\kappa(E_\ep) - \kappa(\Omega)$ in the integrand, apply Cauchy-Schwartz, and use the apriori bounds \eqref{compactness}  to obtain the desired bound $\cO_n(\ep)$ for $\| \kappa(\Omega) - \kappa(E_\ep) \|_{C^n}$. Note that this implies the proposition because $ \| \kappa(E_\ep) -1 \|_{C^n} = \cO_n(\ep)$. 
 
  \end{proof}
\begin{proof}[Proof of Corollary \ref{close 2}]
 
 We first apply a rigid motion to $\d \Omega$ so that it becomes tangent to $\d D$ at $(1,0)$ and stays on the left side of the line $x=1$. We identify the point of tangency $(1, 0)$ with $s=0$. Then the parametrization $\gamma_\Omega$ is uniquely determined by its curvature by the expression
 \begin{equation} \label{gamma} \gamma_\Omega(s) = \left (1-\int_0^s \sin \left (\int _0^{s'} \kappa(\Omega)(s'') \, ds''\right ) \, ds'\, , \, \int_0^s \cos \left (\int _0^{s'} \kappa(\Omega)(s'') \, ds''\right ) \, ds' \right ),\end{equation}
Now let $\theta \in [0, 2\pi]$ be the arc-length parametrization of  the boundary of the unit disk $D$. We write $$\gamma_\Omega(s(\theta)) = (r(\theta) \cos \theta, r(\theta) \sin \theta).$$ 
We want to show that $r(\theta) = 1+ f(\theta)$ with $f(\theta) = \cO_n(\ep)$ in $C^n$ as this is the exactly what the corollary requires.  Here, $s(\theta)$ is given by 
\begin{equation}\label{s of theta} s(\theta) = \int_0^{\theta} \sqrt{r^2(\vartheta) + (r'(\vartheta)) ^2} \, d\vartheta = \int_0^{\theta} \sqrt{(1+f(\vartheta))^2 + (f'(\vartheta)) ^2} \, d\vartheta . \end{equation}

We note that
$$f(\theta) = \| \gamma_\Omega(s(\theta) \| -1. $$
Since by Proposition \ref{close}, we have $\kappa(\Omega)(s) = 1 + h_0(s)$ with $h_0(s)=\cO(\ep)$ in $C^n[0, \ell_\ep]$, by \eqref{gamma} we get
\begin{equation} \label{f} f( \theta) =  \sqrt{ 1+ h_1 (s(\theta))} -1 =h_2(s(\theta)), \end{equation}
with $h_1(s) = \cO(\ep)$ and $h_2(s) = \cO(\ep)$ in $C^n[0, \ell_\ep]$. We emphasize that since $s(\theta)$ depends on $f(\theta)$, we cannot immediately conclude from this equation that $f(\theta) = \cO(\ep)$ in $C^n$.  We proceed with induction. It is clear from \eqref{f} that $\| f \|_{C^0} = \cO(\ep)$. For $f'$ we have
$$ f'(\theta) = \sqrt{(1+f(\theta))^2 + (f'(\theta)) ^2} \; h_2'(s(\theta)).$$
Solving this for $(f')^2$ we get 
 $$ (f'(\theta))^2 = (1+f(\theta))^2 \frac{(h_2'(s(\theta)))^2}{1- {(h_2'(s(\theta))})^2},$$
 which is obviously $\cO(\ep)$ because $h_2 = \cO(\ep)$ in $C^n$ for all $n$. Now assume $\| f\|_{C^{n-1}} = \cO(\ep)$ for some $n \geq 2$. Differentiating \eqref{f} $n$ times,
 $$ f^{(n)}(\theta) =  \frac{ f^{(n)}(\theta) f'(\theta)}{\sqrt{(1+f(\theta))^2 + (f'(\theta)) ^2}} h_2'(s(\theta)) + R_n(\theta),$$
 where the remainder term $R_n$ depends on $f^{(k)}$, $k \leq n-1$, and $h_2^{(k)}$, $k \leq n$. From the induction hypothesis and the form of $R_n$, one can easily see that $R_n = \cO(\ep)$. Solving the above equation for $f^{(n)}$ and using $\|f \|_{C^1} = \cO(\ep)$ and $h_2' = \cO(\ep)$, we conclude the induction and the corollary follows.
\end{proof}

\section{The loop function and its first variation}
This section focuses on the iterations of the billiard map of a nearly circular domain.
Let us first introduce the billiard map and its periodic orbits.

\subsection{\label{BILLIARDSECT} Billiard map and $(p, q)$-periodic orbits}  Consider a $C^\infty$ strictly convex billiard table $\Omega$ with perimeter $\ell$. We parameterize its boundary in the counter-clockwise direction by its arc-length $s$. We define the phase space by $S_{\text{inward}}^* \d \Omega$, i.e. the inward vectors in the unit cotangent bundle of $\d \Omega$. We identify the phase space with
$$\Pi = \R / \ell \IZ  \times [0, \pi],$$ and use $(s, \phi)$ for a point in $\Pi$. Here, $\phi$ represents the angle that the inward unit vector at $s$ makes with the positive unit tangent vector at $s$, i.e. the tangent vectors in the counter-clockwise direction. The billiard map is a smooth twist map on the closed annulus $\Pi$. We write it as 
$$\begin{cases}
	\beta: \Pi \to \Pi,  \\  \beta(s, \phi)= \left( s_1(s, \phi), \phi_1(s, \phi) \right ).
\end{cases}  $$
It is natural and convenient to lift $\beta$ to $\hat \Pi = \IR \times [0, \pi]$. We shall use $(x, \phi)$ for points in $\hat \Pi$. We fix the lift and call it $\hat \beta$ by requiring that $\hat \beta (x, 0) = (x, 0)$. Then by the continuity of the lift we have $\hat \beta (x, \pi) = (x + \ell, \pi)$. We shall write
$$ \hat \beta (x, \phi) = (x_1 (x, \phi), \phi_1 (x, \phi)). $$
The billiard map satisfies the monotone twist property, meaning
$$\d_\phi x_1>0.$$
The map $\hat \beta$ also preserves the orientation and the boundaries of $\hat \Pi$. Moreover, $\hat \beta$ preserves the natural symplectic form $\sin\phi\, dx\wedge d\phi$ on the phase space $\hat \Pi$.  A diffeomorphism  $\beta$ of the annulus $\Pi$  whose lift $\hat \beta$ satisfies the above properties is called a twist map. We refer the readers to \cite{MF} for more on the properties of twist maps and Aubry-Mather theory.

 Note that we can write 
\begin{equation} \label{FG}  \hat \beta (x, \phi) = (x + F(x, \phi), G(x, \phi)), \end{equation}
where $F$ and $G$ are smooth, $\ell$-periodic in the $x$ variable, and $F(x, 0) = G(x, 0) =0$. We shall use $\hat \beta_0$ for the billiard map of the unit disk $D$. One can easily see that 
$$ \hat \beta_0 (x, \phi) = (x+2 \phi, \phi).$$  
A point $(s, \phi) \in \Pi$ is called a $(p, q)$ periodic  point of $\beta$ if $\beta^q(s, \phi) = (s, \phi)$ and the orbit $\{ \beta^{j}(s, \phi) \}_{0\leq j \leq q-1}$ winds $p$ times around $\d \Omega$ in the positive direction. This means that for any lift $(x, \phi)$ of $(s, \phi)$, we have 
$$ \hat \beta^q (x, \phi) = (x+ p \ell, \phi). $$ The ratio $\frac{p}{q}$ is called the rotation number of $(s, \phi)$. Since the rotation number of the time reversal of a periodic orbit of rotation number $\frac{p}{q}$ is given by $\frac{q-p}{p}$, we always assume that $1\leq p \leq \frac{q}{2}$. On the unit disk $E_0$, the $(p, q)$ periodic points form an invariant circle  given by $ \{ (s, \phi)| \, \phi = \pi p / q 
\}$.

\subsection{Expansions of the billiard map} For small angles $\phi$, the billiard map $\hat \beta(x, \phi)$ of a smooth strictly convex domain has a useful expansion (via Taylor's theorem) in the form
$$ \hat \beta (x, \phi) = \left ( x + \sum_{j=1}^{N-1} \alpha_j(x) \phi^j + F_N(x, \phi)\phi^N, \sum_{j=1}^{N-1} \beta_j(x) \phi^j + G_N(x, \phi) \phi^N \right ).$$
By Prop 14.2 of \cite{L}, the remainder terms $F_N$ and $G_N$ are bounded by $ \frac{\cO_N(1)}{\kappa_{\min}} \| \frac{1}{\kappa} \|_{C^{N-1}}$, where $\kappa(x)$ is the curvature at $x$ and $\kappa_{\min}$ is the minimum curvature. 
The coefficients $\alpha_j(x)$ and $\beta_j(x)$ can ba calculated in terms of $\kappa$ and its derivatives; see for example \cite{L} for the expressions of $\alpha_1, \dots, \alpha_4$ and $\beta_1, \dots, \beta_4$.  In fact in certain coordinates, called Lazutkin coordinates, the billiard map can be written in a simpler form. More precisely, if we denote
$$ \xi =  C_1 \int_0^x \kappa^{2/3}(x') dx', \qquad  \eta=  C_2 \kappa^{-1/3}(x) \sin (\phi/2).$$ 
where $C_1= {1 / \int_0^\ell \kappa^{2/3}(x') dx'}$ and $C_2 = 4C_1$, 
then in these coordinates the billiard map takes the form 
$$(\xi, \eta)  \to  \left( \xi + \eta + \sum_{j \geq 3} \tilde{\alpha}_j(\xi) \eta^j  , \eta + \sum_{ j \geq 4} \tilde{\beta}_j(\xi) \eta^j\right),$$
where the coefficient functions are $1$-periodic and smooth. The infinite sums are understood as asymptotic expansions as $\eta \to 0$ and not as convergent power series. 
The remarkable fact is the vanishing of the $\eta^2$ term in the first component and of the $\eta^2$ and $\eta^3$ terms in the second component. In fact as Lemma 14.6 of \cite{L} shows, one can go further inductively and find for each $N \geq 3$, new coordinates $(u, v)$ in which the billiard map is written as
\begin{equation}\label{Lazutkin of order N} (u, v)  \to  \left( u + v + \sum_{j \geq N} a_j(u) v^j  , v + \sum_{ j \geq N+1} {b}_j(u) v^j\right). \end{equation}
Moreover, the construction of this map reveals that the two coordinates $(\xi, \eta)$ and $(u, v)$ are related by
\begin{equation} \label{Lazutkin coordinates}(\xi, \eta) = \left ( u + v^2 A(u, v), v+ v^3 B(u, v)  \right). \end{equation}
Lazutkin coordinates can be very useful is proving existence of invariant curves as shown by Lazutkin but also in studying periodic orbits as done in \cite{DKW}.

\subsection{Nearly circular deformations} Suppose $\Omega$ is nearly circular in $C^1$. Recall that by Definition \ref{NC} this means that $\Omega$ is smooth and simply connected, and can be written as $\d \Omega = \d D + f N_0$, 
where $f$ is a smooth function on $\d D$ sufficiently small in $C^1$ and $N_0$ is the outward unit normal field to $\d D$. This also means that $\d \Omega$ is a polar curve given by $ r(\theta) = 1 + f(\theta)$, $\theta \in [0, 2\pi]$. In fact we will need to consider the linear deformation $\{\Omega_\tau \}_{0 \leq \tau \leq 1}$ defined by 
\begin{equation} \label{deformation}
\d \Omega_\tau = \d D + \tau f N_0, \end{equation}
or equivalently in polar coordinates, by
$$\d \Omega_\tau: \;  r(\tau, \theta) =1 + \tau f(\theta), \quad \theta \in [0, 2\pi].$$
Hence by this notation $\Omega_0 =E_0=D$ and $\Omega_1 = \Omega$. We denote the arc-length parametrization of $\d \Omega_\tau$ by $\gamma(\tau, s)$. The polar and arc-length parametrizations are related by
\begin{equation}\label{s and theta}
 \gamma(\tau, s(\tau, \theta)) = \big ( r(\tau, \theta )  \cos \theta , r(\tau, \theta) \sin \theta \big ), \end{equation}
where
\begin{equation}\label{s} s(\tau, \theta) = \int_0^\theta \sqrt{\left({\d_\theta r} \right)^2(\tau, \vartheta) + r^2(\tau, \vartheta)}  \, d \vartheta = \int_0^\theta \sqrt{ \tau^2 f'^2(\vartheta) + (1 + \tau f(\vartheta))^2} \, d\vartheta. \end{equation}
\begin{rema}
Throughout the paper we identify $s=0$ with $\theta=0$.  
\end{rema}

\subsection{First variations of the deformation}
The first \emph{normal variation} of $\d\Omega_\tau$ at $\tau$ is defined by 
\begin{equation}\label{first variation} n( \tau, s) =  \d_\tau \gamma (\tau, s) \bigcdot N(\tau, s), \end{equation}
where $N(\tau, s)$ is the outward unit normal at $\gamma(\tau, s)$. We also define the first \emph{tangential variation} of $\d \Omega_\tau$ by
\begin{equation}\label{first tangential variation} t( \tau, s) =  \d_\tau \gamma (\tau, s) \bigcdot T(\tau, s), \end{equation}
with $T(\tau, s)$ being the unit tangent in the positive direction. 

The following lemma bounds $n(\tau, s)$ and $t(\tau, s)$ in terms of the function $f(\theta)$. 

\begin{lemm}\label{n and t} Suppose $\|f\|_{C^1} \leq 1$. Then $$ n(\tau, s(\tau, \theta)) = \cO(\|f\|_{C^1}) \qquad \text{and}  \qquad  t(\tau, s(\tau, \theta)) = \cO(\|f\|_{C^1}). $$
\end{lemm}
\begin{proof} We differentiate \eqref{s and theta} and obtain
	\begin{align} \label{d tau} \d_\tau ( \gamma(\tau, s(\tau, \theta))) & = ( f(\theta) \cos \theta, \; f(\theta) \sin \theta), \\
	  \d_\theta ( \gamma(\tau, s(\tau, \theta)) ) & = \left ( \tau f'(\theta) \cos \theta - (1+ \tau f(\theta)) \sin \theta, \; \tau f'(\theta) \sin \theta + (1+ \tau f(\theta)) \cos \theta \right ). \label{d theta}
	  \end{align}
	  In particular from \eqref{d theta} we find that the unit outward normal is given by
	   $$ N ( \tau, s(\tau, \theta)) = \frac{ \left (\tau f'(\theta) \sin \theta + (1+ \tau f(\theta)) \cos \theta, \;  - \tau f'(\theta) \cos \theta + (1+ \tau f(\theta)) \sin \theta \right )}{\sqrt{\tau^2 f'^2(\theta) + (1+ \tau f(\theta))^2}}. $$
	   
	   On the other hand 
	   \begin{align*}(\d_\tau \gamma)(\tau, s(\tau, \theta))  & = \d_\tau ( \gamma(\tau, s(\tau, \theta)))  - \d_\tau s(\tau, \theta)   (\d_s \gamma)(\tau, s(\tau, \theta)),
	   \end{align*}
	   which using $\d_s \gamma \bigcdot N =0$ and $\d_s \gamma \bigcdot T =1$, implies that 
	   \begin{align*}  n(\tau, s(\tau, \theta))) & = \d_\tau ( \gamma(\tau, s(\tau, \theta))) \bigcdot N , \\
	   t(\tau, s(\tau, \theta)) ) & = \d_\tau ( \gamma(\tau, s(\tau, \theta))) \bigcdot T  - \d_\tau s(\tau, \theta).
	   \end{align*}
	   Thus by \eqref{d tau} and \eqref{d theta}, we get
	   \begin{align*} n(\tau, s(\tau, \theta)) & =  \frac{ f(\theta) (1+ \tau f(\theta) )}{\sqrt{\tau^2 f'^2(\theta) + (1+ \tau f(\theta))^2}}, \\
	    t(\tau, s(\tau, \theta)) & = \frac{ \tau f'(\theta) f(\theta)}{\sqrt{\tau^2 f'^2(\theta) + (1+ \tau f(\theta))^2}} - \d_\tau s(\tau, \theta). \end{align*}
	   The lemma then follows easily from $\d_\tau s(\tau, \theta) = \cO(\|f\|_{C^1})$.
	   \end{proof}  

 \subsection{Loop function and its first variation} Our primary purpose in this section is to study the $(1, q)$ periodic orbits of $\Omega_\tau$, hence in particular $\Omega=\Omega_1$, in terms of the ones  of $D$.  The main ingredients will be the loop functions and their linearizations in $\tau$.  We start by the following theorem that introduces what we will call the \textit {loop angle}. 
\begin{theo}\label{loop angle}
Let $\partial \Omega= \partial D + f N_0$ be nearly circular in $C^6$. There exists $\ep_0 >0$ sufficiently small such that if $\| f \|_{C^6} \leq \ep_0$, then for each $\tau \in [0, 1]$,  $s$ on $\d \Omega_\tau$, and $q \geq 2$, there exists a unique angle $\phi _{q} (\tau, s) \in (0, \pi)$ such that the orbit starting at $(s, \phi_q(\tau, s))$ and making $q$ reflections, winds around the boundary once in the counterclockwise direction, and returns to $s$ (not necessarily in the same direction). Moreover, $\phi_q(\tau, s)$ is smooth. We shall call $\phi_q(\tau, s)$ the $q$-loop angle of $\Omega_\tau$. 
\end{theo}

\begin{rema} In \cite{R} and \cite{PR}, a similar statement is proved, however since the proof is based on the implicit function theorem at $\tau =0$, the above theorem  is obtained only for $\tau$ small, and for $\phi$ near $\pi/q$, which is the $q$-loop angle of the disk $\Omega_0$. The size of the neighborhoods of $\tau=0$ and $\phi =\pi /q$ are not estimated in these references. To do this, one probably needs a \textit{quantitative implicit function theorem} (see for example the online notes of Liverani \cite{Li}). In this paper, we take a different route and estimate the $q$-iterations of the billiard map more directly.\end{rema} 
In fact for technical reasons we will need a stronger result as follows. Below, $\ell_\tau$ is the perimeter of $\Omega_\tau$.
\begin{theo}\label{q path}
	Let $\partial \Omega= \partial D + f N_0$ be nearly circular in $C^6$. There exists $\ep_0 >0$ sufficiently small such that if $\| f \|_{C^6} \leq \ep_0$, then for each $\tau \in [0, 1]$,  $s$ and $s'$ on $\d \Omega_\tau$ with $|s-s'| <\frac{\ell_\tau}{100}$, and $q \geq 2$, there exists a unique angle $\alpha _{q} (\tau, s, s') \in (0, \pi)$, such that the orbit starting at $(s, \alpha_q(\tau, s, s'))$ and making $q$ reflections, winds around the boundary approximately once in the counterclockwise direction and ends at $s$.  The function $\alpha_q(\tau, s, s')$ is smooth in $(\tau, s, s')$. 
\end{theo}
Here, by `winding around the boundary approximately once in the counterclockwise direction', we precisely mean that if $x$ and $x'$ are lifts of $s$ an $s'$ with $|x-x'| < \ell_\tau/100$, then $\hat \beta_\tau^q (x, \alpha_q(\tau, x, x')) = x' + \ell_\tau$, where $\hat \beta_\tau$ is the natural lift of the billiard map of $\Omega_\tau$. 

We will give the proof in Section \ref{proof of loop angle}. In the next few pages we draw some important consequences of these theorems.  

Note that by our notations, on the diagonal $s=s'$ we have $\alpha_q(\tau, s, s) = \phi_q(\tau, s)$. We shall call the angle $\phi_q(\tau, s)$ the $q$-loop angle at $s$. We will use $\phi_q(s)$ for the loop angle of $\Omega = \Omega_1$ instead of $\phi_q(1, s)$.  Obviously the loop angle satisfies 
$$ \hat \beta ^q (s, \phi _{q} (\tau, s) ) = (s, \tilde{\phi}_q(\tau, s) ), $$
for some angle $\tilde{\phi}_q(\tau, s)$.  The following lemma is then immediate. 
\begin{lemm} $(s, \phi)$ is a $q$-periodic point of $\d \Omega$ if and only if 
	$$ \phi_q(s) = \tilde{\phi}_q(s) = \phi.$$
	\end{lemm}
We now define the main ingredients of this article, namely the $q$-length function and its first variation. For the rest of the paper we assume that $\|f \|_{C^6}$ is small enough so that Theorems \ref{loop angle} and \ref{q path} hold. 
\begin{def1}[Length and Loop functions] \label{Psi}Let $\Omega$ be $C^6$ sufficiently close to the unit disk $E_0$, and let $q \geq 2$. The $q$-length function $\Psi_q(s, s')$, defined on $|s -s'| <\ell_\tau/100$, is the length of the unique $q$ times reflected geodesic form $s$ to $s'$ defined in Theorem \ref{q path}. The $q$-loop function $L_q(s)$ is the length of the unique $q$-loop at $s$ defined by Theorem \ref{loop angle}, i.e $L_q(s) = \Psi_q(s, s)$.  More precisely, if $\gamma(s)$ is the arc-length parametrization of $\Omega$,
	\begin{equation} \label{length function}
	\Psi_q(s, s') = \sum_{j=0}^{q-1} \| \gamma(s_{j+1}) - \gamma(s_j) \|, \quad s_j = \text{Proj}_1 \beta^j (s, \alpha_q(s, s')),
	\end{equation} 
	where $\text{Proj}_1$ is the projection map onto the base component. Similarly,
	\begin{equation} \label{loop function}
	L_q(s) = \sum_{j=0}^{q-1} \| \gamma(s_{j+1}) - \gamma(s_j) \|, \quad s_j = \text{Proj}_1 \beta^j (s, \phi_q(s)).
	\end{equation}
	Note that since $\alpha_{q}(s, s')$ is smooth, $\Psi_q(s, s')$ and $L_q(s)$ are also smooth. Correspondingly,  we denote the $q$-loop functions of the deformation $\Omega_\tau$ by $L_q(\tau, s)$. 
\end{def1}
The following lemma, although simple, gives a very useful characterization of $q$-periodic orbits in terms of the $q$-loop function. 
\begin{lemm}\label{critical points} We have 
	$$\{ (s, q) \in \Pi ;\, \text{$(s, \phi)$ is $q$-periodic} \} = \{ (s, \phi_q(s));  \,  L_q'(s)=0 \}.$$ 
	In other words, $q$-periodic orbits correspond to the critical points of $L_q$. 
\end{lemm}
	\begin{proof}
Consider the loop $\{(s_j, \vartheta_j)\}_{0 \leq j \leq q}$ (not necessarily a periodic orbit) generated by $(s, \phi_q(s))$, meaning 
$$ (s_j , \vartheta_j)= \beta^j (s, \phi_ q(s)).$$
By this notation $(s_0 , \vartheta_0) = (s, \phi_q(s))$ and $(s_q , \vartheta_q) = (s, \tilde{\phi}_q(s))$. 
Differentiating $L_q$ we get
 \begin{align*}
 L_q'(s) &= \sum_{j=0}^{q-1} \left ( \frac{ds_{j+1}}{ds} \frac{d\gamma}{ds}(s_{j+1}) -  \frac{ds_{j}}{ds} \frac{d\gamma}{ds}(s_{j}) \right ) \bigcdot \frac{\gamma(s_{j+1}) - \gamma(s_j)}{ \| \gamma(s_{j+1}) - \gamma(s_j) \|}
 \\ & = \sum_{j=0}^{q-1} \cos \vartheta_{j+1} \frac{ds_{j+1}}{ds} -  \cos \vartheta_j \frac{ds_{j}}{ds} 
 \\ & = \cos \vartheta_{q} \frac{ds_{q}}{ds} -  \cos \vartheta_0 \frac{ds_{0}}{ds}
 \\ & = \cos \vartheta_{q} -  \cos \vartheta_0.
  \end{align*}
  This shows that $s$ is a critical point of $L_q$ if and only if $\vartheta_0=\vartheta_q$, which by our notation means $\phi_q(s) = \tilde{\phi}_q(s)$. 
\end{proof}
  Next we define the first variation of $L_q$.
  \begin{def1} \label{Melnikov 1} Let $\Omega_\tau$ be a deformation of the unit disk as in Theorem \ref{loop angle} so that the loop angles and loop functions are defined. For each $q \geq 2$,  the first variation of the $q$-loop function is denoted by $ M_q(\tau, s) = \d_\tau L_q(\tau, s). $
  	\end{def1}
  More explicitly we have, 
  \begin{lemm}\label{Melnikov 2} Let $n(\tau, s)$ and $t(\tau, s)$  be the first normal and tangential variations of $\d \Omega_\tau$ defined by \eqref{first variation} and \eqref{first tangential variation}, and let $\{(s_j(\tau, s) , \vartheta_j(\tau,s))\}_{j=0}^{q}$ be the $q$-loop generated by $(s, \phi_q(\tau, s))$, i.e. $ (s_j (\tau, s) , \vartheta_j(\tau, s))= \beta_\tau^j (s, \phi_ q(\tau, s))$, where $\beta_\tau$ is the billiard map of $\Omega_\tau$.  Then 
  \begin{align*}  M_q(\tau, s) = & \; n(\tau, s) \big (\sin \tilde \phi_q(\tau, s) + \sin \phi_q(\tau, s)\big) + t(\tau, s) \big(\cos \tilde \phi_q(\tau, s) - \cos \phi_q(\tau, s)\big) \\ & +2\sum_{j=1}^{q-1} n(\tau, s_j(\tau, s)) \sin \vartheta_j(s, \tau).   \end{align*}
  In particular, when $(s, \phi_q(\tau, s))$ corresponds to a $q$-periodic orbit, i.e. $\phi_q(\tau, s) = \tilde \phi_q(\tau, s)$, we get
   \begin{align*}
   M_q(\tau, s) = 2\sum_{j=0}^{q-1} n(\tau, s_j(\tau, s)) \sin \vartheta_j(s, \tau). 
  \end{align*}
  \end{lemm}

  \begin{proof} For simplicity we use $s_j$ and $\vartheta_j$ for $s_j (\tau, s)$ and $\vartheta_j(\tau, s)$. We then write
  	$$L_q(\tau, s) = \sum_{j=0}^{q-1} \| \gamma(\tau, s_{j+1}) - \gamma(\tau, s_j) \|. $$
  	Taking the variation we get
  	\begin{align*}
  	\d_\tau L_q(\tau, s) = & \sum_{j=0}^{q-1} \left ( \d_\tau s_{j+1} \d_s \gamma(\tau, s_{j+1}) -   \d_\tau s_{j} \d_s \gamma(\tau, s_{j}) \right ) \bigcdot \frac{\gamma(\tau, s_{j+1}) - \gamma(\tau, s_j)}{ \| \gamma(\tau, s_{j+1}) - \gamma(\tau, s_j) \|}
  	\\ &  + \sum_{j=0}^{q-1} \left ( \d_\tau \gamma(\tau, s_{j+1}) -   \d_\tau \gamma(\tau, s_{j}) \right ) \bigcdot \frac{\gamma(\tau, s_{j+1}) - \gamma(\tau, s_j)}{ \| \gamma(\tau, s_{j+1}) - \gamma(\tau, s_j) \|} .
  	\end{align*}
  
Let us denote the two sums by $\Sigma_1$ and $\Sigma_2$, respectively. For $\Sigma_1$, a similar computation as in the proof of Lemma \eqref{critical points} shows that
$$ \Sigma_1   = \sum_{j=0}^{q-1} \cos \vartheta_{j+1} \d_\tau  s_{j+1} -  \cos \vartheta_j  \d_\tau  s_{j}
 = \cos \vartheta_{q} \d_\tau  s_{q} -  \cos \vartheta_0 \d_\tau  s_{0}. 
 $$
 However, because $s_q(\tau, s) = s_0 (\tau, s) =s$, we get $\Sigma_1 =0$. 
 For $\Sigma_2$, we rearrange the sum into 
 \begin{align*} \Sigma_2 = & \sum_{j=1}^{q-1} \left ( \frac{\gamma(\tau, s_{j}) - \gamma(\tau, s_{j-1})}{ \| \gamma(\tau, s_{j}) - \gamma(\tau, s_{j-1}) \|}  -\frac{\gamma(\tau, s_{j+1}) - \gamma(\tau, s_j)}{ \| \gamma(\tau, s_{j+1}) - \gamma(\tau, s_j) \|}   \right )  \bigcdot  \d_\tau \gamma(\tau, s_{j})
 \\ &+ \left ( \frac{\gamma(\tau, s_{q}) - \gamma(\tau, s_{q-1})}{ \| \gamma(\tau, s_{q}) - \gamma(\tau, s_{q-1}) \|}  -\frac{\gamma(\tau, s_{1}) - \gamma(\tau, s_0)}{ \| \gamma(\tau, s_{1}) - \gamma(\tau, s_0) \|}   \right )  \bigcdot  \d_\tau \gamma(\tau, s_{0})
 \end{align*}
 Now let $N(\tau, s)$ and $T(\tau, s)$ be the unit outward normal and unit positive tangent of $\d \Omega_\tau$ at $s$, respectively.  Then
 \begin{align*}
 \Sigma_2 =&  \sum_{j=1}^{q-1} \d_\tau \gamma(\tau, s_j) \bigcdot \big ( N(\tau, s_j) \sin \vartheta_j + T(\tau, s) \cos \vartheta_j + N(\tau, s_j) \sin \vartheta_j - T(\tau, s) \cos \vartheta_j  \big )
 \\ & + \d_\tau \gamma(\tau, s_0) \bigcdot \big ( N(\tau, s_0) \sin \vartheta_q + T(\tau, s_0) \cos \vartheta_q + N(\tau, s_0) \sin \vartheta_0 - T(\tau, s_0) \cos \vartheta_0   \big ), \end{align*} 
 and the lemma follows by noting that $\vartheta_0 = \phi_q(s)$ and $\vartheta_q = \tilde \phi_q(s)$. 
 \end{proof}
 
 The following estimate on $M_q$ will be useful. 
  \begin{lemm}\label{bound on Melnikov} For sufficiently small $\|f\|_{C_2}$, $M_q(\tau, s)  = \cO ( \|f\|_{C^1})$ \end{lemm} 
 \begin{proof}By Lemmas \ref{Melnikov 2} and \ref{n and t}, we get
   \begin{align*}  M_q(\tau, s) = &\left (1 +\sum_{j=1}^{q-1}  \vartheta_j(s, \tau) \right) \cO ( \|f\|_{C^1}).   \end{align*}
   By \eqref{Lazutkin3}, for sufficiently small $\|f\|_{C_2}$, we get $\sum_{j=1}^{q-1}  \vartheta_j(s, \tau)  \leq \frac12l_\tau \max \kappa_\tau \leq 3\pi.$
   \end{proof}

\subsection{Proof of Theorem \ref{loop angle} } \label{proof of loop angle}
Let $\Omega_\tau$, $0 \leq \tau \leq 1$, be as in the previous section, meaning 
$$\d \Omega_\tau = \d E_0 + \tau f(\theta) N_0,$$
where $E_0$ is the unit disk and $f(\theta)$ is a smooth function. We will always use the notation $\ell_\tau =\, \text{length}(\d \Omega_\tau)$. 
Instead of the billiard maps $\beta_\tau$, it is more convenient to use the natural lifts $\hat \beta_\tau$ because they are all defined on the same space $\hat \Pi=\R \times [0, \pi]$.  
To prove Theorem \ref{loop angle}, we need to study the $q$-iterates $\hat \beta^q_\tau$, but before doing this we need the following important perturbative lemma for $\hat \beta_\tau$. We recall that $\hat \beta_0(x,  \phi) = (x+2 \phi, \phi)$. 
\begin{lemm}\label{perturbed beta} Let $|| f||_{C^6} \leq 1$, and $||f||_{C^2}$ be sufficiently small. Then $ \hat \beta_\tau (x, \phi)$ can be written as 
$$\hat \beta_\tau (x, \phi) = ( x + 2 \phi + P_\tau(x, \phi), \phi + Q_\tau(x, \phi)), $$
where $P_\tau$ and $Q_\tau$ are analytic families of $\ell_\tau$ periodic functions in $x$ and 
$$ i, j, k =0, 1: \quad \d_\tau^i \d_x^j \d_\phi^ k P_\tau(x, \phi)=  \phi^{1-k} \cO (||f||_{C^{6}}),$$
$$ i, j, k =0, 1: \quad \d_\tau^i \d_x^j \d_\phi^ k Q_\tau(x, \phi)=  \phi^{2-k} \cO (||f||_{C^{6}}),$$
uniformly for $x \in \R$, $\phi \in [0, \frac{\pi}{2}]$ and $\tau \in [0, 1]$. When $\phi \in [\frac{\pi}{2}, \pi]$, we need to replace $\phi$ with $\pi - \phi$ in the above estimates. 
\end{lemm}
\begin{rema}
We do not claim that $C^6$ is the optimal choice here. In fact, probably  $C^4$ is sufficient, however since $C^6$ is more than good enough for our main theorem, we do not attempt to optimize this lemma. 
\end{rema}
\begin{proof} 
	By symmetry, it is enough to prove the lemma for $\phi \in [0, \frac{\pi}{2}]$, hence we assume this throughout the proof. Let $\hat \beta$ be the lift of the billiard map of a strictly convex domain $\Omega$, with $\hat \beta(x, 0) =(x, 0)$. Also, let $\kappa$ be the curvature function of $\d \Omega$.   By Proposition 14.1 of \cite{L}, we know that if we define 
	$$(x_1, \phi_1) = \hat \beta (x, \phi),$$
	then 
	\begin{equation}\label{Lazutkin1}
	\int_x^{x_1} \sin \left (\phi - \int_x^{x'} \kappa(x'') dx'' \right) dx' =0,
	\end{equation} 
		\begin{equation}\label{Lazutkin2}
		\phi_1 = \int_x^{x_1} \kappa(x') dx' - \phi . 
	\end{equation}  
	Moreover, 
	\begin{equation}\label{Lazutkin3}
\frac{2}{\kappa_{\text{max}}} \phi	\leq x_1 -x  \leq  \frac{2}{\kappa_{\text{min}}} \phi,
	\end{equation}  
		\begin{equation}\label{Lazutkin4}
		\frac{1}{2\kappa_{\text{max}} / \kappa_{\text{min}} -1} \phi	\leq \phi_1  \leq  	\left (2\kappa_{\text{max}} / \kappa_{\text{min}} -1 \right )\phi	.
	\end{equation}  
	Now consider the the deformation $\d \Omega_\tau$, $\tau \in [0,1]$. By our notations,
	$$(x_1, \phi_1) = \hat \beta_\tau(x, \phi) = (x+2\phi + P_\tau(x, \phi), \phi + Q_\tau(x, \phi)). $$
	Since 
	\begin{equation} \label{kappa tau} \kappa_\tau = 1 + \cO ( ||f||_{C^2}), \end{equation} 
		by \eqref{Lazutkin3} we obtain
	\begin{equation} \label{P tau}
	P_\tau(x, \theta) = x_1 -x - 2 \phi = \phi \cO (||f||_{C^2}).
	\end{equation}
	Next as in \cite{L}, we study $P_\tau$ (sometimes we call $P$) as an implicit function defined, using \eqref{Lazutkin1}, by
	\begin{equation} \label{I}
	I(\tau, x, \phi, P) = \int_0^1 \sin \left (\phi - \int_x^{x+t(2\phi + P)} \kappa_\tau(x'') dx'' \right ) dt =0. 
	\end{equation}
	Let us compute $| \d_P I |$ and estimate it from below.  We have
	$$ - \d_P I = \int_0^1  \cos \left (\phi - \int_x^{x+t(2\phi + P)} \kappa_\tau(x'') dx'' \right ) \kappa_\tau \big(x + t(2\phi+P) \big )\,  t \, dt. $$ 
	By \eqref{kappa tau} and \eqref{P tau}, we get 
	\begin{align*} - \d_P I & =  \int_0^1  \cos \left (\phi - \int_x^{x+2t \phi} \, dx'' \right )  t dt + \cO (||f||_{C^2}) \\
	 &=  \int_0^1  \cos (\phi(1-2t))\,  t \, dt + \cO (||f||_{C^2})
	 \\ & = \frac{\sin \phi}{\phi} +  \cO (||f||_{C^2}).
	 \end{align*}
	 Since for $\phi \in [0, \frac{\pi}{2}]$, we have $\frac{\sin \phi}{\phi} \geq \frac{2}{\pi}$, we obtain that for $||f||_{C^2}$ sufficiently small,  uniformly for $0 \leq \tau \leq 1$, $0 \leq \phi \leq \frac{\pi}{2}$, and $x \in \R$, we have 
	 \begin{equation} \label{lower bound dP} | \d_P I | \geq \frac{1}{\pi}. \end{equation}
	 Consequently, by the implicit function theorem there is a unique $P_\tau(x, \phi)$ satisfying \eqref{I} and it is differentiable in $\tau$, $x$ and $\phi$. Moreover, 
	\begin{equation}\label{Implicit derivatives} \begin{cases} \d_\tau P_\tau = -\frac{\d_\tau I }{\d_P I },\\ \\ \d_x P_\tau= - \frac{\d_x I}{\d_P I }, \\ \\
	\d_\phi P_\tau= -\frac{\d_\phi I}{\d_P I}.\end{cases} \end{equation}
	On the other hand since $P_0(x, \phi)=0$, we have $P_\tau(x, \phi) = \tau \tilde P_\tau (x, \phi)$. But since $P_\tau(x, 0) = 0$, we have $\tilde P _\tau (x, 0) = 0$, and so
	$$P_\tau(x, \phi) = \tau \phi R_\tau (x, \phi).$$
	In fact 
	$$R_\tau(x, \phi) = \int_0^1\int_0^1 (\d_\tau \d_\phi P) (u\tau, x, v \phi) du dv. $$
	Here, to ease the notation for the integrand we have set  $P(\tau, x, \phi):= P_\tau(x, \phi)$. 
	
	Similarly for $Q_\tau$ we know that $Q_0(x, \phi) =0$, so $Q_\tau (x, \phi) = \tau \tilde Q_\tau(x, \phi)$. It is known by the asymptotic expansion of the billiard map near $\phi=0$ (see for example page 145 of \cite{L}), that $$Q_\tau (x, 0) = \d_\phi Q_\tau (x, 0) = 0,$$ so we must have 
	$$Q_\tau(x, \phi) = \tau \phi^2 S_\tau (x, \phi). $$
	By the integral remainder formula of the Taylor's theorem, we have
	$$S_\tau(x, \phi) = \frac12 \int_0^1\int_0^1 (1-v) (\d_\tau \d^2_\phi Q)(u\tau, x, v \phi) du dv. $$ 
	Again, here for convenience we have denoted $Q(\tau, x, \phi):= Q_\tau(x, \phi)$. Thus, to prove the lemma it suffices to prove the estimates 
	\begin{equation} \label{P estimates}  \d^k_\phi \d^j_x \d^i_\tau P = \cO (||f||_{C^6}), \quad i=1, 2; \,  j= 0,1;\,  k=1, 2, 3, \end{equation} 
	and 
	\begin{equation} \label{Q estimates} \d^k_\phi  \d^j_x  \d^i_\tau Q = \cO (||f||_{C^6}), \quad i=1, 2;\,  j=0, 1; \,  k=2, 3. \end{equation}
	We first show the estimates for $P$.  An important observation is that for $n \leq 4$, 
	$$ \d_x^n \d_\tau \kappa_\tau (x) = \cO (||f||_{C^{n+2}}). $$
	One then immediately sees by \eqref{I} that 
	$$ \d_\tau^{i} I = \cO(||f||_{C^2}), \quad i=1, 2.$$
	Furthermore,  by taking derivatives with respect to $x$, $\phi$, and $P$, we get
	$$\d_P^m \d_\phi^k \d_x^j\d_\tau^{i} I = \cO(||f||_{C^6}), \quad i=1, 2; j+k+m \leq 5$$
	The estimates \eqref{P estimates} can be concluded, by differentiating the first equation of \eqref{Implicit derivatives}, then using the lower bound \eqref{lower bound dP}, and the above estimates for the derivatives of $I$. 
  
	The estimates for $Q$ follow from the ones for $P$, and the relation
	$$ Q(\tau, x, \phi) = \int_x^{x+2\phi + P(\tau, x, \phi)} \kappa_\tau(x') dx' -2 \phi,$$
	which is obtained from \eqref{Lazutkin2}. 
	 \end{proof}
	
	Equipped with Lemma \ref{perturbed beta},  we are in position to start the proof of Theorems \ref{loop angle} and \ref{q path}. Let $\hat \beta_\tau$ be the lift of the billiard map of $\d \Omega_\tau = \d E_0 + \tau f N_0$ and $\text{Proj}_1$ be the projection onto the $x$ component of $\Pi = \R \times [0, \pi]$.  Our strategy is to show that 
	$$x_q(\tau, x, \phi) := \text{Proj}_1 {\hat \beta_\tau}^q (x, \phi)$$ is strictly increasing as a function of $\phi$ on the interval $[0, C / q]$. More precisely,
	\begin{theo}\label{x_q increasing} For any $C>0$, there exists $\ep_0$ such that for all perturbations $\partial \Omega_\tau = \partial D + \tau f N_0$ of the unit disk $D$ with $\| f \| _{C^6} \leq \ep_0$, for all $q \geq 2$, $ 0\leq \tau \leq 1$,  $x \in \R$,  and $\phi \leq C /q$, we have
		$$\d_\tau x_q(\tau, x, \phi) = q  \phi C^2e^{2C} \cO(\|f\|_{C^6}) ,$$
		$$\d_\phi x_q(\tau, x, \phi) = 2q + q C^2e^{2C} \cO(\|f\|_{C^6}),$$
		and
	$$x_q(\tau, x, \phi) = x+2q \phi + q  \phi C^2e^{2C} \cO(\|f\|_{C^6}).$$ \end{theo}
	
Note that the last statement follows by integrating the the second statement and the fact $x_q(\tau, x, 0) = x $. Before proving this theorem, we need to state and prove a lemma and its corollary. 
	
		\begin{lemm}\label{angle of iterations}
		Let $C>0$ and $\delta >0$. Suppose $\phi \leq C/q$. Then for $\|f\|_{C^6}$ sufficiently small in terms of $C$ and $\delta$, we have that for all $x \in \R$ and $\tau \in [0,1]$, all angles of reflections of the orbit $\{ \hat \beta_\tau^ j (x, \phi) \}_{j=0}^{q}$ are bounded above by $C(1+\delta)/q$. Moreover, if $\phi \geq C/q$, then all angles of reflections are bounded from below by $\frac{C}{(1+\delta)q}$. 
	\end{lemm}

		\begin{proof} Let $\phi_j$ be the $j$-th angle of reflection. We shall prove by induction that for all $0 \leq j \leq q$, we have 
			\begin{equation} \label{phi upper bound} \phi_j \leq \frac{C}{q} \left (1 + \frac{A}{q} \right)^{j},  \end{equation}
			where $A>0$ is chosen such that $e^A = 1+ \delta$. 
			Once proved, this estimate would imply the lemma immediately because 
			$$\frac{C}{q} \left (1 + \frac{A}{q} \right)^{q} \leq \frac{C}{q} e^{A} \leq   \frac{C(1+\delta)}{q}.$$ 
			 The estimate \eqref{phi upper bound} is obviously true for $j=0$ by assumption of the lemma. Assume it is true for some $j \geq 0$. Then using Lemma \ref{perturbed beta}, we have
			\begin{align*}\phi_{j+1} &= \phi_j \left (1 + \phi_j \cO (\|f\|_{C^6}) \right). 
			\end{align*}
			Therefore, the proof is concluded if we choose $\|f\|_{C^6}$ small enough (in terms of $C$ and $\delta$) so that $$\phi_j \cO (\|f\|_{C^6}) \leq \frac{C(1+\delta)}{q} \cO (\|f\|_{C^6}) \leq \frac{A}{q}. $$
			The second statement follows from the first.
			\end{proof}

		As a corollary we have:
\begin{coro} \label{bound on loop angle} For $\| f \|_{C^6}$ is sufficiently small, all angles of reflections of a $q$-reflected path in $\Omega_\tau$ from $x$ to $x'$ with $|x-x'| \leq \frac{\ell_\tau}{100}$, are less than $\frac{3\pi}{2q}$. In particular all angles of reflections of $q$- loops of winding number one are bounded by $\frac{3\pi}{2q}$. 
		\end{coro} 
	\begin{proof}
		By definition of a $q$-reflected path from $x$ to $x'$, we have $x_q (\tau, x, \alpha) =x' +\ell_\tau$. Let $\delta \geq \frac{1}{100}$ to be determined later. Since $\ell_\tau = 2 \pi+ \cO(\|f\|_{C^2})$, if we define 
		$$x_j (\tau, x, \phi) = \text{Proj}_1 {\hat \beta_\tau}^j (x, \phi),$$
		for small enough $\|f\|_{C^2}$ in terms of $\delta$, there must exist $ 0 \leq  j^* \leq q-1$ such that 
		$$x_{j^*+1} - x_{j^*} \leq \frac{x'-x + \ell_\tau}{q} \leq \frac{2(1+\delta)^2 \pi}{q}.$$ By \eqref{Lazutkin3}, for sufficiently small $\|f\|_{C^6}$ (in terms of $\delta$), we get $ \phi_{j^*} \leq \frac{(1+\delta)^3 \pi}{q}$.
	We then start with the point $(x_{j^*}, \phi_{j^*})$ in the phase space and apply the billiard map $\beta_\tau$, $q- j^*$ times, and apply its inverse $\beta_\tau^{-1}$, $j^*$ times. Using Lemma \ref{angle of iterations} we obtain $\phi_j \leq \frac{(1+\delta)^4 \pi}{q}$ which is less than $\frac{3\pi}{2q}$ for example for $\delta = \frac{1}{10}$. 
	\end{proof}

 \begin{proof}[\textbf{Proof of Theorem \ref{x_q increasing}}].  To obtain estimates on the $\phi$ and $\tau$ derivatives  of $x_q(\tau, x, \phi)$ we take its difference with the corresponding function for the unit disk and  differentiate with respect to $\phi$ and $\tau$ and denote it by $A(x, \phi)$, i.e. 
 $$A_q(\tau, x, \phi)= \d_\tau \d_ \phi \left ( \text{Proj}_1 \hat \beta_\tau^q (x, \phi) - \text{Proj}_1 \hat \beta_0^q (x, \phi)\right ). $$
 		 Recall that $ \hat \beta_\tau$ is the billiard map of $\Omega_\tau$ and $ \hat \beta_0$ is the billiard map of the unit disk $D=E_0$. Since 
 		$$\hat \beta_0^q (x, \phi) = (x + 2q\phi, \phi),$$ 
 		we have $$\d_\tau \d_\phi \text{Proj}_1 \hat \beta_0^q (x, \phi) = \d_\tau (2q) =0,$$
 		so in fact
 		$$A_q(\tau, x, \phi)= \d_\tau \d_ \phi \text{Proj}_1 \hat \beta_\tau^q (x, \phi).$$
 	We claim that that for all $\tau \in [0, 1]$, $x \in \R$, and $\phi \in [0, \frac{C}{q}]$,
 $$A_q (\tau, x, \phi) = C^2 e^{2C} q \cO (\| f\|_{C^6}).$$ 
To prove this, we write
\begin{align} \label{D beta expansion}\d_\tau \d_\phi  \text{Proj}_1 \hat \beta_\tau^q (x, \phi ) & = \d_\phi  \text{Proj}_1 \sum_{j=0}^{q-1} {\hat \beta}_\tau^j  \circ ( \d_\tau {\hat \beta}_\tau) \circ {\hat \beta}_\tau ^{q-j-1}(x, \phi) \\
&= \begin{bmatrix} 1 \\ 0 \end{bmatrix}^T \; \sum_{j=0}^{q-1}  D {\hat \beta}_\tau^j  \circ ( D \d_\tau {\hat \beta}_\tau) \circ D {\hat \beta}_\tau ^{q-j-1}(x, \phi) \begin{bmatrix} 0 \\ 1 \end{bmatrix}. 
\end{align}
By Lemma \ref{angle of iterations}, all angles of iterations are bounded by $2C/q$, and by Lemma \ref{perturbed beta}, we have uniformly for all $\tau \in [0, 1]$, $\phi \in [0, \frac{2C}{q}]$, and $x \in \R$,
\begin{equation*} D \hat \beta_\tau (x, \phi)  = \left[ \begin{array}{cc} 1+ \frac{C}{q} \cO ( \|f \|_{C^6}) & 2+ \cO ( \|f \|_{C^6})  \\  \frac{C^2}{q^2} \cO ( \|f \|_{C^6})  & 1 + \frac{C}{q} \cO ( \|f \|_{C^6}) \end{array} \right], \end{equation*}
and 
\begin{equation}\label{d tau beta} D \d_\tau \hat \beta_\tau (x, \phi)  = \left[ \begin{array}{cc} \frac{C}{q} \cO ( \|f \|_{C^6}) & \cO ( \|f \|_{C^6})  \\  \frac{C^2}{q^2} \cO ( \|f \|_{C^6})  & \frac{C}{q} \cO ( \|f \|_{C^6}) \end{array} \right]. \end{equation}
We shall need to estimate the powers of the matrix  $D \hat \beta_\tau$. Breaking it into the diagonal and the off-diagonal part, and factoring the diagonal part,  we see that 
\begin{equation}\label{D beta} | D \hat \beta^j _\tau |  \leq  \left ( 1+ \frac{C}{q} \cO ( \|f \|_{C^6}) \right)^j  \left (I +  \left[ \begin{array}{cc} 0 & 4+ \cO ( \|f \|_{C^6})  \\  \frac{C^2}{q^2} \cO ( \|f \|_{C^6}) & 0  \end{array} \right] \right )^j. \end{equation}
Let us denote 
$$ B = \left[ \begin{array}{cc} 0 & 4+ \cO ( \|f \|_{C^6})  \\  \frac{C^2}{q^2} \cO ( \|f \|_{C^6}) & 0  \end{array} \right] .$$
We note that 
$$ B^2 =  \left[ \begin{array}{cc} \frac{C^2}{q^2} \cO ( \|f \|_{C^6}) & 0  \\  0 & \frac{C^2}{q^2} \cO ( \|f \|_{C^6})  \end{array} \right].$$
By the binomial expansion
\begin{align*} (I + B)^j & = \sum_ {m=0}^j { j \choose m} B^{m} = \sum_ {k=0}^{[j/2]} { j \choose 2k} B^{2k} 
+ \sum_ {k=0}^{[(j-1)/2]} { j \choose 2k+1} B^{2k+1}. \end{align*}
Thus 
$$ (I + B)^j \leq  \sum_ {k=0}^{[j/2]} { j \choose 2k}  \frac{C^{2k}}{q^{2k}} \cO ( \|f \|^k_{C^6})
 I + \sum_ {k=0}^{[(j-1)/2]} { j \choose 2k+1} \frac{C^{2k}}{q^{2k}} \cO ( \|f \|^k_{C^6})  B .$$
 We choose $ \|f \|_{C^6}$ small enough so that all $\cO ( \|f \|^k_{C^6})$ terms are bounded by one. To estimate the second sum we note that for any $a>0$:
 $$ \sum_{k=0}^{[(j-1)/2]}{ j \choose 2k+1} a^{2k} = \frac{(1 + a)^j - (1-a)^j}{2a} \leq j (1+a)^{j-1}. $$
 Hence, 
 $$ (I + B)^j \leq \left[ \begin{array}{cc} \left (1 + \frac{C}{q} \right)^j & 10 j \left (1 + \frac{C}{q} \right)^{j-1}  \\  \frac{C}{q} \left (1 + \frac{C}{q} \right)^j  &  \left (1 + \frac{C}{q} \right)^j  \end{array} \right] .$$
 Plugging this into \eqref{D beta}, we get 
 $$ | D \hat \beta^j _\tau | \leq \left(1 + \frac{C}{q} \right )^{2j} \left[ \begin{array}{cc} 1 & 10 j   \\  \frac{C}{q}  &  1  \end{array} \right] . $$ 
 Inserting this estimate into \eqref{D beta expansion} and using \eqref{d tau beta}, we arrive at
 \begin{align*} | \d_\tau \d_\phi & \text{Proj}_1 \hat \beta_\tau^q (x, \phi )|  \\ & \leq \left (1+ \frac{C}{q} \right )^{2q} \; \sum_{j=0}^{q-1} \begin{bmatrix} 1 \\ 0 \end{bmatrix}^T    \left[ \begin{array}{cc} 1 & 10 j   \\  \frac{C}{q}  &  1  \end{array} \right]  \left[ \begin{array}{cc} \frac{C}{q} & 1   \\  \frac{C^2}{q^2}  &  \frac{C}{q}  \end{array} \right]  \left[ \begin{array}{cc} 1 & 10 (q -j -1)   \\  \frac{C}{q}  &  1  \end{array} \right]      \begin{bmatrix} 0 \\ 1 \end{bmatrix} \cO( \| f \|_{C^6}) \\
 & \leq C^2 e^{2C} q \cO( \| f \|_{C^6}). \end{align*}
 The theorem follows by integrating this with respect to $\phi$ and $\tau$ separately, and the facts $x_q(\tau, x, 0) =x$ and $x_q(0, x, \phi) = x +2q$. 
  \end{proof}
For future reference we record that our estimates also show the following bounds for the $\phi$ derivative of $\vartheta_q(\tau, x, \phi) = \text{Proj}_2 \hat \beta_\tau^q (x, \phi)$, where $\text{Proj}_2$ is the projection onto the second component.  
\begin{lemm}\label{vartheta derivative} $\d_\phi \vartheta_q(\tau, x, \phi) = 1+ C^2 e^{2C} \cO( \| f \|_{C^6})$. 
\end{lemm}

\subsubsection{\textbf{Concluding the proof of Theorems \ref{loop angle} and \ref{q path}}}

Theorem \ref{x_q increasing} shows that $x_q(\tau, x, \phi)$ is monotonically increasing in $\phi$ on $[0, 3\pi/2q]$ and $$x_q(\tau, x, \phi) = x+ 2q \phi + q \phi \cO (\|f \|_{C^6}).$$ Since $x_q(\tau, x, 0) =0$, and because for sufficiently small $\|f \|_{C^6}$, 
$$x_q (\tau, x, 3\pi/2q) = x + 3 \pi + \cO (\|f \|_{C^6}) > x' + \ell_\tau,$$
by the intermediate value theorem there must be a unique $\phi = \alpha_q(\tau, x, x') \leq 3\pi/2q$ such that 
$$ x_q (\tau, x, \alpha_q(\tau, x, x')) = x'+ \ell_\tau.$$
By the implicit function theorem, $\alpha_q(\tau, x, x')$ is smooth in $(\tau, x, x')$. This together with Corollary \ref{bound on loop angle} conclude the proof of Theorem \ref{q path}, thus also Theorem \ref{loop angle}.  
  
  \section{Length spectrum} 
  Let $\Omega$ be a smooth strictly convex domain. 
  For $1 \leq p \leq \frac{q}{2}$, we denote $\mathcal{L}_{p, q}(\Omega)$ to be the set of lengths of periodic orbits of type $(p, q)$, i.e. periodic orbits that make $q$ reflections and wind around the boundary of $\Omega$, $p$ times, in the counterclockwise direction.  
  The length spectrum of $\Omega$ is 
  $$ \mathcal{L}(\Omega) = \text{closure}  \bigcup_{ 1 \leq p\leq q/2} \mathcal{L}_{p, q}(\Omega) $$We also denote $T_{p, q}$ and $t_{p, q}$ to be the sup and inf of $\mathcal{L}_{p, q}(\Omega)$, respectively. Marvizi-Melrose \cite{MM} proved that for a fixed $p$, as $q \to \infty$, one has
  \begin{equation}\label{lengths difference} T_{p, q} - t_{p,q} = \cO (q^{- k}), \end{equation}
  and moreover, there are constants $c_{k, p}(\Omega)$ such that
  \begin{equation} \label{T asymptotic} T_{p, q} \sim p \, \ell + \sum_{k=1}^\infty c_{k, p}(\Omega) q^{-2k}, \end{equation}
  where $\ell$ is the perimeter of $\Omega$. Note that this in particular shows that natural integer multiples of $\ell$ belong to the length spectrum as they are limit points of closed geodesics. 
  
  The Mather function $\ss(\omega)$ is a strictly convex function on $[0, \frac12]$ (see for example \cite{MF, Sib04}) whose values at the rational numbers are given by 
  \begin{equation}\label{beta} \ss(p/q) = - \frac{1}{q}T_{p, q}. \end{equation}
 The following lemma will be useful for us. From now on we shall use $T_q = T_{1,q}$ and $t_q= t_{1,q}$. 
  \begin{lemm}\label{gaps decreasing} The sequence $\{T_{q}\}_{q \geq 2}$ is strictly increasing to $\ell$, and its gap sequence $$\{T_{q+1} - T_{q} \}_{q \geq 2},$$ is strictly decreasing. 
  	\end{lemm}
  \begin{proof} Since $\ss$ is strictly convex, the slopes must strictly increase on its graph, hence
  	$$ \frac{\ss \left(\frac{1}{q+1} \right ) - {\ss} \left(\frac{1}{q+2} \right )} {\frac{1}{q+1} - \frac{1}{q+2} } < \frac{\ss \left(\frac{1}{q} \right ) - \ss \left(\frac{1}{q+1} \right )} {\frac{1}{q} - \frac{1}{q+1} }.$$ The lemma follows quickly by \eqref{beta}. \end{proof}
  \begin{rema} One can try to prove this using the asymptotic \eqref{T asymptotic}, however this method would only prove the lemma for large $q$. \end{rema}
  		
  For our purposes we  will need the following rough, but quantitative, version of estimates  \eqref{lengths difference} and  \eqref{T asymptotic}. 
  \begin{lemm} \label{MM} Let $\d \Omega_\tau = \d E_0 +  \tau f N_0$ be a nearly circular deformation in $C^6$. Assume $\|f\|_{C^8} \leq 1$ and $\|f\|_{C^2}$ is sufficiently small so that $\kappa_\tau = 1+ \cO(\|f\|_{C^2}) \geq \frac12$. Then uniformly for $\tau \in [0,1]$ we have:
  	 \begin{equation}\label{lengths difference v2} T_{q} - t_{q} = q^{-3} \cO (\|f\|_{C^6}) + \cO(q^{-4}) , \end{equation}
  	 \begin{equation} \label{T asymptotic v2} T_{q}  = \ell_\tau - \frac{1}{4}\left (\int_{0}^{\ell_\tau} \kappa_\tau^{2/3}(s) \, ds\right)^3 q^{-2}  +q^{-3} \cO (\|f\|_{C^6}) + \cO(q^{-4}).
  	  \end{equation}
  	  Here, the constants in all $\cO$ remainders are universal. 
  	\end{lemm}
  \begin{proof} We provide a proof using ``Lazutkin coordinate" and the Euler-Maclaurin formula.  Note that this is a quantitative version of a result in \cite{MM} which was not obtained in this reference. 
  
It is sufficient to prove this lemma for $\tau=1$. We shall use $\Omega$ for $\Omega_1$, $\kappa$ for $\kappa_1$, and $\ell$ for $\ell_1$.  We first recall the Lazutkin coordinate which is a diffeomorphism from $\mathbb R / \ell \mathbb Z$ to  $\mathbb R / \mathbb Z$ defined by
\begin{equation} \label{Laz}  \xi = \frac{\int_0^s \kappa^{2/3}(s') ds'}{\int_0^\ell \kappa^{2/3}(s') ds'}. \end{equation} 
Here we have used $\xi$ instead of $x$, the later being the standard notation for the Lazutkin coordinate, to avoid confusion with our $x$ used for the lift of $s$ to $\mathbb R$. The periodic orbits of type $(1, q)$ in $\Omega$ have a rather nice description in the Lazutkin coordinate. To present this feature, let $\{ (s_j, \phi_j)\}_{j=0}^{q-1}$ be any such periodic orbit and let $\{\xi_j \}_{j=0}^{q-1}$ correspond to $\{s_j\}_{j=0}^{q-1}$ in Lazutkin coordinate. From Appendix A of \cite{DKW}, one can show that there exists a $1$-periodic smooth function $\alpha$ defined only in terms of $\kappa$, with $\|\alpha \|_{C^m} = \cO ( \| (1/\kappa)' \|_{C^{m+1}})$, such that
\begin{equation} \label{xi} 
\xi_j = \xi_0 + \frac{j}{q} + \frac{\alpha \left (\xi_0+ {j}/{q} \right )} {q^2} +  \frac{\cO(\| (1/\kappa)' \|_{C^3} )} {q^4}, \qquad  1 \leq j \leq q-1. 
\end{equation} 
We note that by our lower bound assumption $\kappa \geq \frac12$, we can replace the term $ \cO(\| (1/\kappa)' \|_{C^3})$ in the remainder by $\cO(\| \kappa' \|_{C^3}) = \cO(\|f \|_{C^6})$.  We shall use \eqref{xi} to find an asymptotic for the length $T$ of the orbit $\{s_j\}_{j=0}^{q-1}$. Let $s(\xi)$ be the inverse function of $\xi = \xi(s)$ defined by \eqref{Laz} and let $\tilde \gamma (\xi)= \gamma ( s (\xi))$.  We write 
$$ T= \sum_{j=0} ^{q-1} \| \gamma(s_{j+1}) - \gamma(s_j) \| =  \sum_{j=0} ^{q-1} \| \tilde  \gamma(\xi_{j+1}) - \tilde \gamma(\xi_j) \|. $$ 
Inserting \eqref{xi} and using the mean value theorem,
$$ T=   \sum_{j=0} ^{q-1} \left \| \tilde  \gamma \left (\xi_0 + \frac{j+1}{q} + \frac{\alpha \left (\xi_0+{(j+1)}/{q} \right )} {q^2} \right) - \tilde \gamma \left (\xi_0 + \frac{j}{q} + \frac{\alpha \left (\xi_0+{j}/{q} \right )} {q^2} \right) \right \| +   \frac{\cO(\|f \|_{C^6} )} {q^3}. $$
For the sum we shall use the Euler-Maclaurin formula which asserts that if $g(\xi)\in C^\infty[0,1]$ and $g^{(k)}(0 )= g^{(k)}(1)$ for all $ k \geq 0$, then for all $m \geq 1$
$$ \frac{1}{q} \sum_{j=0}^{q-1} g (j/q) = \int_0^1 g(\xi) d \xi + R_m(g),$$
with $$ | R_m(g) | \leq \frac{2\zeta(m)}{(2\pi q)^{m}} \int_0^1 | g^{(m)}(\xi) | d\xi. $$
In our situation, we have
$$g(\xi) = \left \| \tilde  \gamma \left (\xi_0 + \xi + \frac{1}{q} + \frac{\alpha (\xi_0+  \xi + 1/q)} {q^2} \right) - \tilde \gamma \left (\xi_0 + \xi + \frac{\alpha (\xi_0+\xi)}{q^2} \right) \right \|, $$
which is a smooth $1$-periodic function on $[0, 1]$. Thus if we choose $m=4$ we obtain
$$ T= q \int_{0}^1 \left \| \tilde  \gamma \left (\xi + \frac{1}{q} + \frac{\alpha ( \xi + 1/q)} {q^2} \right) - \tilde \gamma \left (\xi + \frac{\alpha (\xi)}{q^2} \right) \right \| d\xi + \frac{\cO(1+\|f\|_{C^8})}{q^4} +  \frac{\cO(\|f \|_{C^6} )} {q^3}. $$ 
Taylor expanding the integrand we arrive at
$$ T = a_0 + \frac{a_1}{q} + \frac{a_2}{q^2} + \frac{a_3}{q^3} + \frac{\cO(\|f\|_{C^6})}{q^3} +  \frac{\cO(1+\|f\|_{C^8})}{q^4}.$$
It is clear that $a_0 = \int_0^1 \| \tilde \gamma ' (\xi) \| d \xi = \ell$. That $a_1 =0$ and $a_2 =- \frac{1}{4}\left (\int  \kappa^{2/3}\right)^3$ follows from \cite{MM}. Since by \eqref{T asymptotic} only even powers of $q^{-1}$ appear in the expansion, we must have $a_3 = \cO(\|f \|_{C^6})$, so
$$T = \ell + \frac{a_2}{q^2} + \frac{\cO(\|f\|_{C^6})}{q^3}+  \frac{\cO(1+\|f\|_{C^8})}{q^4}.$$
\subsection{A proof independent of \cite{DKW}} The proof of \eqref{xi} in  \cite{DKW} is given under the assumption of  axial symmetry.  However, this assumption is not  essential, even for the proof in \cite{DKW}. For the sake of independence,  we present another proof of Lemma \ref{MM},   at the mild cost of replacing  $\| f\|_{C^6}$ with $\| f\|_{C^7}$, and $\| f\|_{C^8}$ with $\| f\|_{C^9}$ in the lemma. 

 For the proof of  Lemma \ref{MM} we only need a weaker statement than \eqref{xi}, which states that there exist $1$-periodic smooth functions $\alpha_1$ and $\alpha_2$ defined only in terms of $\kappa$, with $\|\alpha_1 \|_{C^m} = \cO ( \| (1/\kappa)' \|_{C^{m+1}})$, $\|\alpha_2 \|_{C^m} = \cO ( \| (1/\kappa)' \|_{C^{m+2}})$ such that
\begin{equation} \label{xi v2} 
\xi_j = \xi_0 + \frac{j}{q} + \frac{\alpha_1 \left (\xi_0+ {j}/{q} \right )} {q^2} +  \frac{\alpha_2 \left (\xi_0+{j}/{q} \right )} {q^3}  + \frac{\cO(\| (1/\kappa)' \|_{C^3} )} {q^4}, \qquad  1 \leq j \leq q-1. 
\end{equation} 
We emphasize that the results of \cite{DKW} imply that  $\alpha_2 \equiv 0$, but we do not need this stronger statement. Assuming \eqref{xi v2}, the proof of Lemma \ref{MM} follows line by line as before but at the end we would obtain,
$$T = \ell + \frac{a_2}{q^2} + \frac{\cO(\|f\|_{C^7})}{q^3}+  \frac{\cO(1+\|f\|_{C^9})}{q^4}.$$
So if throughout we assume near circularity in $C^9$, the rest of the proof goes through without modification.

 To prove \eqref{xi v2}, we shall use Lazutkin coordinates of order 5 as defined in \eqref{Lazutkin of order N}. Let $(u, v)$ be such coordinates. Then the billiard map is of the form
 \begin{equation} \label{Lazutkin 5}\beta(u, v) = \big (u + v + v^5 a(u, v), v + v^6 b(u, v) \big), \end{equation} for smooth functions $a$ and $b$ that are bounded by $\cO (\| (1/ \kappa)' \|_{C^3} )$. These bounds follow from similar bounds for the derivatives of the billiard map
 in $(s, \phi)$ coordinates, and also from the construction of the coordinates $(u,v)$ 
 from $(s, \phi)$, which are provided in \cite[Prop 14.2 \& Lemma 14.6]{L}.

 Let $(u_0, v_0)$ correspond to a $(1, q)$ periodic orbit, meaning $\hat \beta^q (u_0, v_0) = (u_0 +1, v_0)$. By \eqref{Lazutkin 5} and finite induction we obtain, $$u_0 + q v_0  + q \cO(|v_0|^5 \| (1/ \kappa)' \|_{C^3} ) = u_0+1.$$
 This implies that 
 $$ v_0 +  \cO(|v_0|^5 \| (1/ \kappa)' \|_{C^3} ) = \frac1q, $$
 which in turn gives $ v_0 = \frac{1}{q}  + \frac{\cO(\| (1/ \kappa)' \|_{C^3} )}{q^5}.$

 Now for  each $1 \leq  j \leq q$, let $(u_j, v_j) = \hat \beta^j (u_0, v_0)$. Then by finite induction on $j$, we can show that
 $$ u_j = u_0+ \frac{j}{q}  +  \frac{\cO(\| (1/ \kappa)' \|_{C^3} )}{q^4}, \qquad v_j = \frac{1}{q}  + \frac{\cO(\| (1/ \kappa)' \|_{C^3} )}{q^4}. $$
 In the final step, we recall from \eqref{Lazutkin coordinates}, that $(\xi, \eta)$ and $(u, v)$ are related by
 $$(\xi, \eta) = \left ( u + v^2 A(u, v), v+ v^3 B(u, v)\right),$$
 where $A$ and $B$ are smooth and $\|A \|_{C^m} = \cO ( \| (1/\kappa)' \|_{C^{m+1}})$, $\|B \|_{C^m} = \cO ( \| (1/\kappa)' \|_{C^{m+2}})$ . If we write 
 $$ A(u, v) = A_0(u) + A_1(u) v + \cO(\| (1/\kappa)' \|_{C^3})|v|^2,$$ then 
 $$ \xi_j = u_j + v_j^2 A(u_j, v_j) = u_0 + \frac{j}{q} + \frac{A_0 \left (u_0+ {j}/{q} \right )} {q^2} +  \frac{A_1 \left (u_0+{j}/{q} \right )} {q^3}  + \frac{\cO(\| (1/\kappa)' \|_{C^3} )} {q^4}.$$
 Writing $u_0$ in terms of $\xi_0$ and $\eta_0$, we obtain \eqref{xi v2}.

 \end{proof}

  	We now focus on the part of the length spectrum that is less than the length of the boundary. While this does not inclusively correspond to $(1,q)$ periodic orbits (consider a very thin ellipse for example), but as we show it does for nearly circular domains.
  	\begin{lemm}\label{winding twice or more} Let $\Upsilon$ be a $(p, q)$ periodic orbit with $p \geq 2$ of a nearly circular deformation $\Omega_\tau$ of $D$ in $C^6$. Then for $\|f\|_{C^6}$ sufficiently small, the length of $\Upsilon$ is bounded below by $\ell_\tau$, uniformly for $\tau \in [0,1]$. Hence in particular, 
  	\begin{equation} \mathcal{L}(\Omega_\tau) \cap (0, \ell_\tau) = \bigcup_{q\geq 2} \mathcal{L}_{1, q}(\Omega_\tau). \end{equation}
  		\end{lemm}
  	
  	
  	\begin{proof}[Proof of Lemma \ref{winding twice or more}] Proposition 5 of \cite{A2} would imply this lemma easily, however the proof of \cite{A2} is not correct. Hence we give an independent proof by means of variations.
  	
  	To clarify the idea, we first verify Lemma \ref{winding twice or more} in the case of the  unit disc $D$. It is well-known 
  	that every link of a billiard trajectories of $D$ (not necessarily a periodic trajectory)  intersects the boundary with the same angle of incidence  say $\phi$. One can then easily verify that length of each link is $2\sin \phi$, thus if the trajectory makes $q$ bounces its length must be $2q \sin \phi$.  The angle
  	for a $(p, q)$ periodic orbit on a circle is given by $\phi= \frac{p}{q}\pi$. Therefore, the length of a $(p, q)$ orbit on the circle is $2 q \sin \frac{p}{q} \pi$.  Since, $\sin x \geq \frac{2}{\pi} x$ on the interval $[0, \pi /2]$, we get $2q \sin \frac{p}{q} \pi \geq 4 p  > 2 \pi$ for $p \geq 2$.
  	
  	Now let $\partial \Omega_\tau = \partial D + \tau f N_0$ be a nearly circular deformation of the unit circle $\partial D$ with $f$ sufficiently small in $C^6$. We wish to approximate the lengths of $(p, q)$ periodic orbits of $\d \Omega_\tau$ by the ones of the disk $\d D$ using a variational method.  Let $(x_0, \phi_0)$ be an initial point in the phase space of $\partial \Omega_{\tau}$ of a $(p, q)$ periodic orbit of $\beta_\tau$ with $p \geq 2$. We keep in mind that $q \geq 4$ because $\frac{p}{q} \leq \frac12$.  Since we can choose $\phi_0 \in [0, \pi /2]$ (otherwise consider $\pi - \phi_0$), there is a unique $k_0 \in \{0, 1, 2, \dots , q \} $  such that
  	$$ \frac{(k_0-\frac12) \pi}{2q} \leq  \phi_0 \leq \frac{(k_0 + \frac12) \pi}{2q} .$$
  	We first claim that $k_0 \geq 4$. To prove this let $(x_j(\tau), \phi_j(\tau)) = \hat \beta_\tau^j (x_0, \phi_0)$ with $0 \leq j \leq q$. Since $(x_0, \phi_0)$ is a $(p, q)$ periodic orbit we have $x_q(\tau) - x_0 = p \ell_\tau$. Suppose $k_0 \leq 3$. Then by  Lemma \ref{angle of iterations}, for a given $\delta >0$ we get $\phi_j(\tau) \leq  \frac{ 7\pi(1+\delta)}{4q}$ for sufficiently small $f$ in $C^6$. In particular $$ \sum_{j=0}^{q-1} \phi_j(\tau) \leq \frac{7 \pi(1+\delta)}{4}. $$
  	On the other hand by \eqref{Lazutkin3} this sum must be larger than or equal to 
  	$$ \frac{1- \delta}{2} ( x_q(\tau) -x_0) = \frac{p \ell_\tau(1-\delta)}{2} \geq 2 \pi (1-\delta)^2, $$
  	which leads to a contradiction for $\delta$ small.  
  	
  	From now on we assume $k_0 \geq 4$. We consider the (partial) orbit $\{(x_j(\tau), \phi_j(\tau))\}_{j=0}^{q_0}$ in $\Omega_\tau$ where we define 
  	\[ q_0 := \left \{ \begin{array}{ll}   \left[ \frac{4q}{k_0} \right] +1 & \text{if} \quad k_0 \geq 5 \\
  	q & \text{if} \quad k_0= 4  \end{array} \right. . \]
  
  	 One can easily verify, using $q \geq k_0 \geq 4$, that $q_0$ never exceeds $q$. As a result, a lower bound for the length of this partial orbit provides a lower bound on the length of the desired full orbit  $\{(x_j(\tau), \phi_j(\tau))\}_{j=0}^{q}$. So let us set
  	$$ b(\tau) = \sum_{j=0}^{q_0-1} \| \gamma(\tau, x_{j+1}(\tau))  -  \gamma(\tau, x_j(\tau)) \|. $$ Obviously for the case of the unit disk $D$, i.e. $\tau=0$, we have $b(0) = 2q_0 \sin \phi_0$. By the mean value theorem, for all $\tau \in [0, 1]$, 
  	\begin{equation} \label{b} |b(\tau) - 2q_0 \sin \phi_0|  \leq \sup_{\tau \in [0, 1]} | \d _\tau b(\tau) |. \end{equation}
  	 We compute the variation of $b(\tau)$ in the same manner as in the proof of Lemma \ref{Melnikov 2}. We first write
  	\begin{align*}
  	\d_\tau b_\tau = & \sum_{j=0}^{q_0-1} \left ( \d_\tau x_{j+1} \d_s \gamma(\tau, x_{j+1}) -   \d_\tau x_{j} \d_s \gamma(\tau, x_{j}) \right ) \bigcdot \frac{\gamma(\tau, x_{j+1}) - \gamma(\tau, x_j)}{ \| \gamma(\tau, x_{j+1}) - \gamma(\tau, x_j) \|}
  	\\ &  + \sum_{j=0}^{q_0-1} \left ( \d_\tau \gamma(\tau, x_{j+1}) -   \d_\tau \gamma(\tau, x_{j}) \right ) \bigcdot \frac{\gamma(\tau, x_{j+1}) - \gamma(\tau, x_j)}{ \| \gamma(\tau, x_{j+1}) - \gamma(\tau, x_j) \|} .
  	\end{align*}
  	
  	Denote the two sums by $\Sigma_1$ and $\Sigma_2$, respectively. For $\Sigma_1$, a similar computation as in the proof of Lemma \eqref{critical points} shows that
  	$$ \Sigma_1   = \sum_{j=0}^{q_0-1} \cos \phi_{j+1}(\tau) \d_\tau  x_{j+1}(\tau) -  \cos \phi_j(\tau)  \d_\tau  x_{j}(\tau)
  	= \cos \phi_{q_0}(\tau) \d_\tau  x_{q_0}(\tau). 
  	$$
  	For $\Sigma_2$, following the proof of Lemma \ref{Melnikov 2}, we obtain
  	\begin{align*}
  	\Sigma_2 = & \d_\tau \gamma(\tau, x_0) \bigcdot \big ( N(\tau, x_0) \sin \phi_0 - T(\tau, x_0) \cos \phi_0 \big )+ \d_\tau \gamma(\tau, x_q) \bigcdot \big ( N(\tau, x_q) \sin \phi_q + T(\tau, x_q) \cos \phi_q \big ) \\
  	& + 2 \sum_{j=1}^{q_0-1} \d_\tau \gamma(\tau, x_j) \bigcdot  N(\tau, x_j) \sin \phi_j 
  	\\ =& n(\tau, x_0)\sin \phi_0  - t(\tau, x_0) \cos \phi_0 + n(\tau, x_{q_0}(\tau)) \sin \phi_{q_0}(\tau) + t(\tau, x_{q_0}(\tau)) \cos \phi_{q_0}(\tau) \\ & +2\sum_{j=1}^{q_0-1} n(\tau, x_j(\tau)) \sin \phi_j(\tau).   \end{align*}
  	
  	Let us estimate $| \Sigma_1|$ and $|\Sigma_2|$.  Since 
  	$$\phi_0 \leq \frac{(k_0 + \frac12) \pi}{2q} \leq \frac{4\pi}{q_0-1}, $$
  	by Theorem \ref{x_q increasing} we have 
  	$$ | \Sigma_1| \leq | \d_\tau x_{q_0}(\tau) | \leq C_1 q_0 \phi_0 \|f \|_{C^6} \leq 8 \pi C_1 \|f \|_{C^6}, $$
  	where $C_1$ is universal constant. 
  	
  	To estimate $|\Sigma_2|$ we first observe that by Lemma \ref{n and t}, each $n(\tau, x)$ and $t(\tau, x)$ is of size $\cO ( \|f \|_{C^1})$.  On the other hand by \eqref{Lazutkin3} and Theorem \ref{x_q increasing}, for $\|f\|_{C^6}$ small enough, we get
  	$$ \sum_{j=1}^{q_0 -1}  | \sin \phi_j(\tau) |  \leq  \sum_{j=1}^{q_0 -1}  \phi_j(\tau) \leq \frac{1+\delta}{2} (x_{q_0}(\tau) - x_0)  \leq 2 (1+\delta) \, q_0 \phi_0  \leq 32 \pi. $$
  	Revisiting the inequality \eqref{b}, we have just proved that
  	$$ b(\tau) = 2q_0 \sin \phi_0 + \cO ( \|f\|_{C^6}).$$ 
  	Since $\phi_0 \geq \frac{(k_0 -\frac12)\pi}{2q}$, we get
  	$$ b(\tau) \geq  \frac{4q_0}{ \pi } \phi_0 + \cO ( \|f\|_{C^6}) \geq 8 - \frac{4}{k_0}+ \cO ( \|f\|_{C^6}) \geq 7 + \cO ( \|f\|_{C^6}), $$
  	which is strictly larger than $\ell_\tau = 2 \pi + \cO(\|f\|_{C^1})$ for $ \|f\|_{C^6}$ sufficiently small. 
  		\end{proof}

  	In the next section we study the $(1, q)$ length spectrum in terms of the $q$-loop function and its first variation. 
  	
  \subsection{Length spectrum and loop functions}
  Let $\Omega_\tau$ be as in the previous sections. It is then obvious from the definition of $q$-loop fucntion $L_q(\tau, s)$ that
  \begin{equation} \mathcal{L}_{1, q} (\Omega_\tau) = \text{critical values of} \; L_q(\tau, s). \end{equation}
  Thus, in particular
  $$ T_q(\tau) = \max L_q(\tau, s), \quad t_q(\tau) = \min L_q(\tau, s). $$
  The following lemma  will be very useful for small values of $q$. 
  \begin{lemm}\label{T-t} For all $q \geq 2$, we have uniformly in $\tau \in [0,1]$:
  	$$ L_q(s) = 2q \sin(\pi /q) +  \cO ( \|f\|_{C^6}).$$
  	In particular,
  	$$ T_q(\tau) - t_q(\tau) = \cO ( \|f\|_{C^6}).$$
  	 \end{lemm}
   \begin{proof} By the mean value theorem and the definition \eqref{Melnikov 1} of $M_q$, we have
   	\begin{equation*} \left | L_q(\tau, s) - L_q(0, s) \right |  \leq  \sup_{\tau \in [0,1]} | \d_\tau L_q(\tau, s)| =  \sup_{\tau \in [0,1]} |M_q(\tau, s)|. \end{equation*}
   However, for the unit disk the loop function is constant. In fact $L_q(0, s) = 2q\sin ( \pi/q)$. 
   Also, 
   $$ T_q(\tau) - t_q(\tau) = \max L_q(\tau, s) -  \min L_q(\tau, s) \leq 2 \sup_{\tau \in [0,1]} |M_q(\tau, s)|. $$
   The lemma follows quickly from Lemma \ref{bound on Melnikov}.
\end{proof}

   Let us now state a key structural result for the $q$-length spectrum.
   \begin{lemm}\label{structure of lengths} Let $\d \Omega_\tau = \d D + \tau f N_0$ be an $\ep$-nearly circular deformation in $C^6$ i.e. $\|f\|_{C^6} = \cO(\ep)$. Assume in addition $\|f \|_{C^8} \leq 1$. Then for $\ep$ sufficiently small, the $q$-length spectra $\mathcal{L}_{1, q}(\Omega_\tau)$ are disjoint for distinct values of $q \geq 2$. Moreover, there exists $q_0$ uniform in $\tau$ and $f$ such that
   	\begin{itemize} 
   	\item[(a)] For $q \geq q_0$: 
   	\begin{equation} \label{consecutive} t_{q+1}(\tau) - T_q(\tau) = \min \mathcal{L}_{1, q+1}(\Omega_\tau) -   \max \mathcal{L}_{1, q}(\Omega_\tau) \geq \frac{1}{10(q+1)^3}, \end{equation}
   	 \begin{equation}\label{max min difference} T_q(\tau) - t_q(\tau) = \max \mathcal{L}_{1, q}(\Omega_\tau) -   \min \mathcal{L}_{1, q}(\Omega_\tau) \leq \frac{1}{100(q+1)^3}. \end{equation}
   	\item[(b)]  For $2 \leq q \leq q_0$, we have for sufficiently small $\|f\|_{C^2}$ that is uniform in $\tau$,
   	\begin{equation} \label{consecutive 2} t_{q+1}(\tau) - T_q(\tau) = \min \mathcal{L}_{1, q+1}(\Omega_\tau) -   \max \mathcal{L}_{1, q}(\Omega_\tau) \geq \frac{\eta_0}{10}, \end{equation} 
   	\begin{equation}\label{max min difference 2} T_q(\tau) - t_q(\tau)=\max \mathcal{L}_{1, q}(\Omega_\tau) -   \min \mathcal{L}_{1, q}(\Omega_\tau) \leq \frac{\eta_0}{100}, \end{equation}
   	where 
   	\begin{equation} \label{eta_0}\eta_0 =2(q_0+1)\sin ( \pi/(q_0+1) ) - 2q_0\sin ( \pi/q_0 ) \end{equation} 
   	\end{itemize}
   	\end{lemm}
   
   \begin{rema} We note that the estimate \eqref{max min difference} is rather rough. In fact one has estimates of the form $\cO (q ^{-N})$, but this would force us to use more derivatives of $f$ which will be unnecessary for our purposes. 
   	\end{rema}
   	
   	\begin{proof} It is obvious from Lemma \ref{MM} that we can find a universal $q_0$ (hence in particular uniform in $\tau$), so that estimates \eqref{consecutive} and \eqref{max min difference} hold true. In fact we choose $\|f\|_{C^6}$ small enough and $q_0$ large enough so that the remainder terms in Lemma \ref{MM} satisfy
   \begin{equation}\label{q_0}q^{-3} \cO (\|f\|_{C^6}) + \cO(q^{-4}) < \frac{1}{100} (q+1)^{-3}. \end{equation} Estimate \eqref{max min difference 2} follows from Lemma \ref{T-t}, by choosing $\|f\|_{C^2}$ small enough in terms of the universal constant $\eta_0$. It only remains to prove \eqref{consecutive 2}. For this, we note that using Lemma \ref{T-t}, we have
   	$$t_{q+1}(\tau) - T_q(\tau) = 2(q+1)\sin ( \pi/(q+1) ) - 2q\sin ( \pi/q ) + \cO (\|f \|_{C^2}).$$
   	On the other hand by Lemma \ref{gaps decreasing} and the definition \eqref{eta_0} of $\eta_0$, for all $2 \leq q \leq q_0$ we have
   	 $$2(q+1)\sin ( \pi/(q+1) ) - 2q\sin ( \pi/q ) \geq \eta_0.$$ 
   	\end{proof}
   Let us state a very interesting corollary of this lemma. It shows that for nearly circular domains the number of bounces $q$ can be heard from the length spectrum.
   
  \begin{coro}\label{q can be heard} Let $\Omega_1$ and $\Omega_2$ be two nearly circular domains in $C^6$ satisfying conditions of Theorem \ref{structure of lengths}. Suppose 
  	\begin{equation} \label{Omega 1 and 2} \mathcal L (\Omega_1) \cap (0, 2 \pi +1/10) = \mathcal L (\Omega_2) \cap (0, 2 \pi +1/10). \end{equation} Then $\ell(\d \Omega_1) = \ell(\d \Omega_2)$ and for all $q \geq 2$,
  	$$ \mathcal L_{1, q}(\Omega_1) = \mathcal L_{1, q}(\Omega_2). $$ 
  	\end{coro}
  We comment that instead of $2 \pi +1/10$ one can use $2 \pi +\delta $ for any $\delta>0$, but the smallness of $f$ in $C^6$ would depend on $\delta$. 
  
  \begin{proof} As we saw in the proof of Lemma \ref{winding twice or more}, the length of every periodic orbit of type $(p\geq 2, q)$ must be larger than $7 + \cO(\|f\|_{C^6})$.   Let $\d \Omega_1 = \d D + f_1 N_0$ and $\d \Omega_2 = \d D + f_2 N_0$. We choose $f_1$ and $f_2$ small enough so that 
  	$$ 2\pi +1/10 <  7 + \cO(\|f_1\|_{C^6}) \leq  \inf \cup_{p\geq 2} \mathcal L_{p, q}(\Omega_1),  \qquad 2\pi +1/10 <  7 + \cO(\|f_2\|_{C^6}) \leq \inf \cup_{p\geq 2} \mathcal L_{p, q}(\Omega_2).$$
  	In addition we also choose them small enough such that
  	 $$ \ell(\d \Omega_1) =2\pi + \cO(\|f_1\|_{C^1}) < 2\pi+1/10, \qquad   \ell(\d \Omega_2) =2\pi + \cO(\|f_2\|_{C^1}) < 2\pi+1/10.$$
  	 Then under these conditions if we take supremum of \eqref{Omega 1 and 2}, we obtain $\ell(\d \Omega_1) = \ell(\d \Omega_2)$. Now by Lemma \ref{winding twice or more}, we get
  	 $$\bigcup_{q\geq 2} \mathcal{L}_{1, q}(\Omega_1) =\bigcup_{q\geq 2} \mathcal{L}_{1, q}(\Omega_2). $$
  	 Let us denote this common union by $U$. Let $q_0$ and $\eta_0$ be as in Lemma \ref{structure of lengths}. Recall that they are identical for $\Omega_1$ and $\Omega_2$. We first show by finite induction that for each $q \leq q_0$ one has $\mathcal{L}_{1, q}(\Omega_1) =\mathcal{L}_{1, q}(\Omega_2)$. Clearly $t_2(\Omega_1) = t_2(\Omega_2)$ because this is the infimum of the union $U$. We then move in $U$ starting at $t_2$ and record all gaps. By Lemma \ref{structure of lengths}, the first gap that is larger than or equal $\eta_0/10$ must take place at $T_2(\Omega_1)$ and $T_2(\Omega_2)$, hence these two quantities must agree. We then take infimum of $U \cap (T_2, \ell)$ to obtain $t_3(\Omega_1) = t_3(\Omega_2)$. By continuing this procedure we get that for all $2 \leq q\leq q_0$,
  	 \begin{equation} \label{tT} t_q(\Omega_1)= t_q(\Omega_2) \quad \text{and} \quad T_q(\Omega_1)= T_q(\Omega_2).\end{equation}
  	 The argument for $q \geq q_0$ is very similar. We shall use induction. Obviously \eqref{tT} holds for $q=q_0$. Suppose now that \eqref{tT} holds for some $\tilde q >q_0$. Then by taking inf of $U \cap (T_{\tilde q}, \ell)$ we obtain $t_{\tilde q +1}(\Omega_1) = t_{\tilde q +1}(\Omega_2)$. By Lemma \eqref{structure of lengths} again, starting at $t_{\tilde q +1}$ the first gap of size at least $(1/10) (\tilde q +2)^{-3}$ happens at $T_{\tilde q +1}(\Omega_1)= T_{\tilde q +1}(\Omega_2)$. Thus \eqref{tT} holds for all $q \geq 2$ and the corollary follows. 
  	\end{proof}
   
   \subsection{Length spectrum of a nearly circular ellipse}
   Let $E_\ep$ be an ellipse of eccentricity $\ep$. We choose $\ep$ small enough so that no periodic orbits of type $(p, q)$ with $p \geq2$ contribute to the part of the length spectrum that is less than $\ell_\ep$, the perimeter of $E_\ep$. This is possible by Lemma \eqref{winding twice or more}. Thus we will only focus on the $q$-spectrum, i.e. $\mathcal{L}_{1,q}$. Since the ellipse is completely integrable, for each $q \geq 3$, all $(1,q)$ periodic orbits have the same length. Therefore for all $q \geq 3$ we have $T_q = t_q$, or in other words the $q$-loop functions of the ellipse collapse to a constant.  In the case $q=2$, it is known that the only $(1, 2)$ periodic orbits are the bouncing ball orbits on the major and minor axes, whose lengths correspond to $T_2$ and $t_2$, respectively. Note that $T_2 \neq t_2$ if the ellipse is not a disk. In summary, 
   \begin{equation} \label{length spectrum of ellipse}
   \mathcal{L}(E_\ep) \cap (0, \ell_\ep) = \big \{ t_2(\ep) \leq T_2(\ep) < T_3(\ep) < \cdots T_q(\ep) < T_{q+1}(\ep) < \cdots \big \} ,
   \end{equation}
   and the gaps sequence $\{T_{q+1}(\ep) - T_q(\ep)\}_{q=2}^\infty$ is strictly decreasing by Lemma \ref{gaps decreasing}.

\section{Wave trace and Marvizi-Melrose parametrices}

   \subsection{Background on wave trace}  Suppose $\Omega$ is a smooth planar domain. Let
   $$w_\Omega(t) = \text{Tr} \cos ( t \sqrt{\Delta_\Omega})$$
   be the wave trace of $\Delta_\Omega$, the positive Laplacian associated to $\Omega$ with Dirichlet (or Neumman) boundary condition. Let also  $\text{SingSupp}\, w_\Omega(t)$ denote the singular support of $w_\Omega(t)$. By a result of Andersson-Merlose \cite{AnMe} (in the convex case) and Petkov-Stoyanov \cite{PS} (for general smooth domains), we have the so called Poisson relation
   \begin{equation} \label{AM} \text{SingSupp}\, w_\Omega(t) \subset -\mathcal{L}(\Omega) \cup \{ 0\} \cup \mathcal{L}(\Omega).  \end{equation}
   The wave trace is the distribution integral, 
   $$w_{\Omega}(t) = \int_{\Omega} E(t,x,x)d x, $$
   where $E(t, x, y) $ is the Schwartz kernel of $\cos t \sqrt{\Delta_{\Omega}}$ with the prescribed Dirichlet or Neumann
   boundary conditions. By distribution integral is meant that the integral $w_{\Omega}(\rho):= \int_{\R} w_{\Omega}(t){\rho}(t) dt$ is a temperate distribution. By Poisson relation, $w_{\Omega}(t)$ is a $C^{\infty}$ function on the complement of the length spectrum. The nature of its singularities at the lengths depends on the structure of the
   periodic billiard trajectories of $\Omega$. Only  convex smooth plane domains, and only lengths of periodic trajectories  in $ (0, |\partial \Omega| )$ are 
   relevant to this article, so we restrict our attention to them. 
   
   In general, the singularities of $w_{\Omega}(t)$ at $t \in \lcal(\Omega)$ can be extremely complicated since the 
   set of lengths and the set of periodic orbits themselves can be extremely complicated. The simplest periodic orbits
   are the non-degenerate ones, i.e. isolated non-degenerate  critical points of the length function on the  configuration space $(\partial \Omega)^q$ of $q$ points  minus the `diagonals' where two points are equal. The major and minor 
   axes of the ellipse are of this type. The next simplest are
   `clean fixed point sets', i.e. a smooth curve of periodic orbits of length $L$ satisfying the cleanliness condition (see \cite{GM2}), such as the periodic orbits of an ellipse (except the major and minor axes). Equivalently, {\it clean} is in the Bott-Morse sense that for each $q$, the fixed
   point set of $\beta^q$ is a submanifold of $B^*\partial \Omega$, and the
   tangent space to the fixed point set is the fixed point set of $(d \beta)^q$. In general, the set of periodic
   orbits of length $L$ may be as complicated as the critical point set of a smooth function.  
   
   As discussed in \cite[(6.8)]{MM}, there exists a Lagrangian (i.e. oscillatory integral)  parametrix $E_{\epsilon}'(t, x, y)$ for the wave kernel
   $E(t,x, y)$ away from the boundary and  modulo a  $C^{\infty} $ error,  $w_{\Omega}(t) $ may be expressed as the trace of this parametrix away from the boundary plus an additional boundary region  term. We will not use this expression directly and
   refer to \cite[(6.8)]{MM} for background. An alternative oscillatory integral formula based on layer potentials and boundary integral operators  was introduced by Balian-Bloch and  exploited in  \cite{Ze} (to which we refer for background and references).  
   
   In the case of non-degenerate transversally reflecting  periodic billiard trajectories, or clean curves of periodic billiard
   trajectories, of length L,  there exists a microlocal parametrix due to Chazarain which one can use to calculate the singularities
   of $w_{\Omega}(t)$ near $t = L$. In \cite{GM2}, Guillemin-Melrose review  this parametrix construction and use it to derive a singularity
   expansion for $w_{\Omega}(t)$ in these cases; see \cite[Theorem 1-Theorem 2]{GM2}. The non-degeneracy (or cleanliness) assumptions make it possible to apply stationary phase (on a manifold with boundary) to the integral $\int_{\R} \hat{\rho}(t) e^{i t \lambda} w_{\Omega}(t) dt$; see \cite[Lemma 5.2]{GM2}.
   
   In the generic case where all periodic orbits of length $< L(\partial \Omega)$ are non-degenerate transversal  periodic reflecting
   rays, the  length spectrum is discrete in $[0, L(\partial \Omega))$ and accumulates only  at $L(\partial \Omega)$.  The  wave trace admits a decomposition, on $\R^{\geq 0}$, into terms with singular support at a single length $L \in \mathcal L(\Omega)$:
   \begin{equation} \label{WT}w_{\Omega}(t) =  \hat{\sigma}(t) =  \begin{array}{ll} e_0(t) + \sum_{L \in \mathcal L (\Omega)} e_L(t), & {\rm{Sing Supp}} \;e_L = \{L\}.
   \end{array}  \end{equation} 
   The term $e_0(t)$ is singular only at $t  =0$ and admits the asymptotic expansion,
   $$
   e_0(t) =  C_{n} |\Omega| \Re \{(t + i0)^{-n +1} \} + C_{n-1}  |\partial \Omega| \Re \{(t + i 0)^{-n -\half}\} + \text{lower order terms}, $$
   in terms of homogeneous Lagrangian singularities decreasing in singularity by unit steps.  When the billiard flow of $\Omega$ has {\it clean fixed point sets}, 
   \begin{equation} \label{e gamma} e_L(t) = \sum_{\gamma: L_{\gamma} = L} e_{\gamma}(t), \end{equation}
   where the last sum is over components of the closed billiard trajectories of length $L$ or equivalently over components of the fixed points of iterates of the billiard map. In this case, 
   the terms $e_{\gamma}(t)$ admit singularity expansions depending on the dimension $d_{\gamma}$ of the component of the  fixed point set. They have the form
   \begin{equation} \label{WTgamma} e_{\gamma}(t) =  \Re \{ a_{\gamma, 1} (t-L + i 0^+)^{- n_ \gamma} \} + \text{lower order terms}, \end{equation}
   where the
   exponent $n_\gamma$ (the `excess') equals $1 + d_\gamma/2 $.   Since
   $\dim B^*\partial \Omega = 2$, the cleanliness means that either the fixed
   point set consists of isolated non-degenerate fixed points or else of smooth
   curves of transversally non-degenerate fixed points. 
   
   Much of the difficulty of inverse spectral theory is caused by multiplicities in the length spectrum. In the non-degenerate or clean case, the Poisson formula \eqref{WT} then expresses a singularity at $t=L$ as the  sum of contributions from all closed orbits
   of length $L$. Since the coefficients are signed, the terms in this sum may cancel. 
   
   We do not use the Poisson formula \eqref{WT}, because
   cleanliness is not even a generic condition, and does not hold for the fixed point set of
   $\beta^q$  for general almost
   circular domains. Instead, we  follow the idea of  Marvizi-Melrose \cite{MM} to break up $w_{\Omega}(t)$ into
   a sum of $q$-bounce contributions, and then to express the q-bounce contribution as an oscillatory integral
   whose phase is, roughly speaking, the q-bounce loop length function for loops of winding number one. In the nearly circular case, we can prove
   that the Melrose-Marvizi type  parametrices are valid for all bounce numbers $q \geq 3$. Since the closed trajectories do not necessarily
   form clean sets, it  is not generally possible to apply stationary
   phase to these oscillatory integral. But we can use them effectively in the inverse problem.

   \subsection{Marvizi-Melrose Parametrices for nearly circular domains}
 Let $\Omega$ be a nearly circular domain in $C^6$, that is $\d \Omega = \d E_0 + f(\theta) N_0$ for some smooth function $f(\theta)$ on the unit circle $\d E_0$ with $\| f\|_{C^6}$ sufficiently small. 
   
By Lemmas \ref{structure of lengths} and \ref{winding twice or more} for $\tau =1$, the  intervals $[t_q, T_q]$ are disjoint from each other for $q \geq 2$, and are also disjoint from $[t_{p,q}, T_{p,q}]$ for all $p>1$ and $q \geq 2$. Hence we can choose a cutoff function $\hat{\chi}_q(t) \in C^\infty_0(\R)$ that equals one on the interval $[t_q, T_q]$, whose support does not contain any lengths in $ \mathcal{L}_{p, m}$ with $ (p, m)\neq (1, q)$.
 We then denote as in \cite{MM}, 
	\begin{equation} \label{sigma} \hat{\sigma}_{1,q}(t) = \hat{\chi}_q(t) w_\Omega(t). \end{equation}
	The distribution $\hat{\sigma}_{1,q}(t)$ is the localization of the wave trace to the interval $[t_q, T_q]$; it satisfies 
	$$ \text{SingSupp}  \hat{\sigma}_{1,q}(t) \subset [t_q, T_q],$$ 
	and
	$$ \hat{\sigma}_{1,q}(t) = w_\Omega(t) \quad \text{near $[t_q, T_q]$}.$$
	Thus instead of studying the wave trace, we study the localized wave trace $\hat{\sigma}_{1,q}(t)$.

	Marvizi-Melrose \cite{MM} proved that (in fact for any smooth strictly convex domain): 
	\begin{theo}[Proposition 6.11 of \cite{MM}]
	For $q \geq q_0(\Omega)$ sufficiently large, one has a parametrix of the form 
	\begin{equation} \label{trace parametrix} \hat{\sigma}_{1,q}(t) = \int_0^\infty \int_{\d \Omega} \Re \left ({e^{i \pi r_q/4} e^{i\xi(t - L_q(s))} \xi^{\frac12} a(q, t, s, \xi}) \right) ds d \xi  + R_q(t), \end{equation}
	where $R_q(t)$ is a smooth function, $L_q(s)$ is the $q$-loop function, $r_q$ is a Maslov index that depends on $q$ and the boundary condition\footnote{The indices $r_q$ for Dirichlet and Neumman boundary conditions differ from each other by $4 q$.}, and $a(q, t, s, \xi)$ is a smooth classical symbol in $\xi$, and periodic in $s$, of the form
	$$a(q, t, s, \xi) \sim \sum_{j=0}^\infty a_j(q,t, s) \xi^{-j} , \qquad (\xi \to +\infty),$$ whose principal symbol $a_0(q, t, s)=a_0(q, s)$ is independent of $t$ and is a positive function on $\d \Omega$. 
\end{theo}
		\begin{rema}In \cite{MM}, the factor $\xi^\frac12$ is missing from the integrand of \eqref{trace parametrix}. In \cite{P}, there is instead a factor $\xi$. None of these are correct. The correct factor in the principal term must be $\xi^{\frac12}$ as one can easily inspect by the wave trace asymptotic  of Guillemin-Melrose \cite{GM1} of simple non-degenerate periodic orbits. An independent proof of this parametrix, given by Vig \cite{Vig2}, also confirms the factor  $\xi^{\frac12}$. \end{rema}

We prove that the Marvizi-Melrose parametrix \eqref{trace parametrix} is valid for all $q \geq 2$ for nearly circular domains.
	\begin{theo} \label{parametrix theorem} Suppose $\Omega$ is nearly circular in $C^8$ meaning that $\d \Omega = \d E_0 +f N_0$ with $\|f\|_{C^8}$ sufficiently small. Then the parametrix $\eqref{trace parametrix}$ for the wave trace $w_\Omega(t)$ is valid for all $q \ge 2$. 
		\end{theo}
\begin{proof} As in \cite{MM}, the first step in the proof is the reduction of the wave trace to a boundary integral for which we provide a simple proof independent of \cite{MM}. We shall use a Rellich-type identity (essentially Green's 2nd identity). Let $X$ be a suitable vector field defined in a neighborhood of $\Omega$. Let $\phi_j$ be an $L^2$-normalized eigenfunction of $\Delta =- \d^2_x - \d^2_y$, with Dirichlet boundary condition on $\Omega$, associated to the eigenvalue $\lambda^2_j$. First, by Green's second identity we have:
		\begin{equation} \langle \phi_j, [\Delta , X]\phi_j \rangle_{L^2(\Omega)} = \int_\Omega (\Delta -\lambda_j^2 )\phi_j \, X \phi_j  - \phi_j  X ( \Delta - \lambda_j^2) \phi_j
		+ \int_{\d \Omega} \d_n \phi_j \, X \phi_j - \phi_j \, \d_n (X \phi_j).\end{equation} 
		But since $ ( \Delta - \lambda_j^2) \phi_j =0$ and $\phi_j|_{\d \Omega}=0$, we get 
		$$\langle \phi_j, [\Delta , X]\phi_j \rangle_{L^2(\Omega)} = \int_{\d \Omega} \d_n \phi_j \, X \phi_j.$$ Now if we choose the vector field $X = x \d_x + y \d_y$, we get $[\Delta, X] = 2 \Delta$, which implies that 
		\begin{equation}\label{Hassell?} 2 \lambda_j^2 =  \int_{\d \Omega}  (X \cdot n) | \d_n \phi_j|^2. \end{equation}
		Here, the function $(X.n)(s)$ is the dot product of the position vector $X(s)$ on the boundary with the outer unit normal $n(s)$ to $\d \Omega$ at $s$. Using this equation the second time derivative of the trace of $E(t)= \cos (t \sqrt{\Delta_\Omega})$ can be written as (cf. \eqref{WTDEF}): 
		\begin{equation} \label{Hassell wave trace}
		\d_t^2  w_{\Omega}(t)= - \frac12 \int_{\d \Omega} (X \cdot n) \d_{n_x} \d_{n_y} E(t, x, y)|_{x=y}, \end{equation}
		where $E(t, x, y)$ is the distributional kernel of $E(t)$ and $\d_{n_x}, \d_{n_y}$ are outer normal derivatives in the variables $x, y$, respectively. Equation \eqref{Hassell wave trace} is our reduction to the boundary. We now need to find a parametrix for the distribution $\d_{n_x} \d_{n_y} E(t, x, y)|_{x=y}$ and plug it into \eqref{Hassell wave trace}. This distribution was studied thoroughly  in \cite{HeZeE}. Away from the tangential directions to $\d \Omega$, it is a Lagrangian distribution whose wavefront relation satisfies: 
		$$ \text{WF}' (\d_{n_x} \d_{n_y} E(t, x, y)) \subset \bigcup_{ q \geq 0} \Gamma_{q, \pm}^\d,$$
		where 
		\begin{equation}\label{Gamma q} \Gamma_{q, \pm}^\d = \left \{ (t, \pm \tau, x, \xi, y, \eta) \in T^*(\R \times \d \Omega \times \d \Omega) \left\lvert \begin{array}{ll}  \tau >0, \; |\xi| < \tau, \; |\eta| < \tau \\ \tilde{\beta}^q \left (x, \frac{\xi}{\tau} \right) = \left (y, \frac{y}{\eta}\right),\; t= t\left(q, x, \frac{\xi}{\tau} \right) \end{array}  \right. \right\}. \end{equation}
		Here, $\tilde \beta$ is the billiard map as a map on the ball bundle $B^* \Omega = \{ (x, \xi) \in T^* \d \Omega: \; |\xi| \leq 1 \}$ instead of $\Pi = \d \Omega \times [0, \pi]$. Note that $\tilde \beta: B^* \Omega \to B^* \Omega$ and $\beta: \Pi \to \Pi$ are related by
		$$ \tilde \beta (x, \xi) = \beta ( x, \arccos(\xi)).$$
		Furthermore, in \eqref{Gamma q}, $t(q, x, \xi/\tau)$ is defined as follows. Let $\xi' \in T_x^* \bar{\Omega}$ be the unique inward unit vector whose orthogonal projection onto the cotangent space of the boundary $T_x^* \d\Omega$ is $\xi/\tau$. Then $t(q, x, \xi/\tau)$ is the time that it takes to travel from $(x, \xi')$ and make precisely $q$ reflections along the billiard trajectory and of course end at $(y, \eta')$. 
		
		Since for nearly circular domains $ \text{SingSupp}  \hat{\sigma}_{1,q}(t) \subset [t_q, T_q],$ we restrict our attention to a (time) neighborhood of $[t_q, T_q]$.  We recall that we are only interested in a parametix near the diagonal. These two conditions force the billiard trajectories from $x$ to $y$ to wind around the boundary almost once. By our Theorem \ref{q path}, for each $q \geq 2$ there is a unique such trajectory whose length we denoted by $\Psi_q( x, y)$ (see Def \ref{Psi}). In particular, for $(t, \pm \tau, x, \xi, y, \eta) \in \Gamma_{j, \pm}^\d$ and $t \in [t_q, T_q]$, we have  $t(q, x, \xi / \tau) = \Psi_q( x, y)$. In fact an easy computation (see for example \cite{Vig1, Vig2}) shows that $ \Psi_q(x, y)$ generates $\tilde \beta^q$, in the sense that,
		$$ \tilde \beta^q (x , -\d_x\Psi_q ) = (y, \d_y \Psi_q).$$
		Now let us define the phase functions, 
		$$\Phi_{\pm, q}:= \pm \tau( t - \Psi_q(x, y)).$$ The critical point set of $\Phi_\pm$ with respect to $\tau$ is given by $$C(\Phi_{\pm, q}) = \{(t, \pm \tau, x, y)| \; t = \Psi_q(x, y) \}.$$
		Hence the image of $C(\Phi_{\pm, q})$ under the canonical embedding 
		$$\iota (t, \tau, x, y) = (t, \d_t \Phi_{\pm, q}, x, \d_x \Phi_{\pm, q}, y, -\d_y \Phi_{\pm, q}) = (t, \pm \tau,  x, \mp \tau \d_x \Psi_q, y, \pm \tau \d_y \Psi_q),$$
	is precisely $\Gamma_{q, \pm}^\d$ defined by \eqref{Gamma q}. In other words, $\Phi_{\pm, q}$ parameterizes $\Gamma_{q, \pm}^\d$. By \cite[Prop 4]{HeZeE}, the principal symbol of $ \d_{n_x} \d_{n_y} E(t, x, y)$ on $\Gamma_{q, \pm}^\d$, as a half-density, is of the form
		\begin{equation} \label{Symbol} e^{i \pi m_q /4} g(q, x, \xi, \tau) | dx \wedge d \xi \wedge d \tau|^{1/2}, \end{equation} where $m_q$ is a Maslov index and $g$ is a certain positive symbol. Now  by \cite[IV, Prop 25.1.5]{Ho}, and because $\d_{n_x} \d_{n_y} E(t, x, y)$ is a real-valued Lagrangian distribution,  there exists a classical symbol $ b(q, t, x, y, \tau)$ such that for $t \in [t_q, T_q]$ and $(x,y)$ near the diagonal of the boundary,  
	$$ \d_{n_x} \d_{n_y} E(t, x, y) = \Re \int_0^\infty e^{i \pi m_q /4} e^{i \tau( t - \Psi_q(x, y))} b (q, t, x, y, \tau) d \tau + \text{(smooth)}.$$
	
	By expanding $b$ into a power series of $t - \psi_q$, and performing integration by parts, we can in fact eliminate the $t$ variable from the amplitude $b$. The principal term $b_0$ of $b$ can be calculated in terms of $g$ in \eqref{Symbol}. Since $g >0$, we have $b_0>0$. 
	Plugging this parametrix into \eqref{Hassell wave trace} and noticing that $\Psi_q(x, x)= L_q(x)$, we obtain a parametrix for $\d_t^2 w_{\Omega}(t)$ in the form:
	$$ \d_t^2 w_{\Omega}(t) |_{[t_q, T_q]}= \int_0^\infty \int_{\d \Omega} \Re \left (e^{i \pi r_q/4} e^{i\tau(t - L_q(s))} \tau^{\frac12}  \tilde b(q, s, \tau) \right) ds d \tau  + (\text{smooth}),$$
	where $r_q = m_q+4$ and
	$$  \tau^{\frac12}\tilde b(q, s, \tau) = \frac12  b(q, s, s, \tau) (X\cdot n)(s).$$
	Because $b_0>0$ and because $X. n >0$ for a strictly convex domain (if the origin is chosen to be inside $\Omega$), the principal term $\tilde b_0$ must be positive as well. 
	Now let 
$$ \tilde w_{\Omega}(t):=  \int_0^\infty \int_{\d \Omega} \Re \left (e^{i \pi r_q/4} e^{i\tau(t - L_q(s))} \tau^{\frac12}  a(q, s, \tau) \right) ds d \tau, $$
with amplitude $a(q, s, \tau):= -\tau^{-2}  \tilde b(q, s, \tau)$. It is then clear that 
$\d_t^2 (w_{\Omega}(t)  -  \tilde {w}_\Omega(t))$ is smooth on  $[t_q, T_q]$, thus $w_{\Omega}(t) = \tilde w_{\Omega} (t) + (\text{smooth}),$ and the result follows. 
	
		\end{proof}
	
\begin{rema}
	The article \cite{Vig2} calculates the principal term $a_0$ explicitly. Also, it is possible to use the
	layer potential formulae of \cite{Ze} for the semi-classical resolvent kernel to prove Theorem \ref{parametrix theorem}.  Since this proof is much longer and more complicated than the proof above, we do not present it here.
\end{rema}
	
	\subsection{$\hat{\sigma}_{1,q}$ as a spectral invariant: Proof of
	Proposition \ref{MMSPINV}}
	
	In this section we state more formally the following: 
	\begin{prop}\label{MMSPINVb} If $\Omega$ is a nearly circular domain in $C^8$, then the distribution $\hat{\sigma}_{1,q}(t)$ is a spectral invariant
	of $\Omega. $ \end{prop}

The key ingredient in the proof is Corollary \ref{q can be heard}.  The wave trace on the open interval $(0, |\partial \Omega|) = (0, \ell)$  has the following {\it q-bounce} decomposition, introduced in \cite{MM}:
\begin{equation} \hat{\sigma}(t) |_{(0, \ell)}= \sum_{q \geq 2} \hat{\sigma}_{1, q}(t).\end{equation}
By Corollary \ref{q can be heard}, $\lcal_{1,q} \cap \lcal_{1, q'} =\emptyset
$ if $q \not= q'$ and $\lcal_{1,q} \cap \lcal_{p,q'} = \emptyset$ if $p \geq 2$.
Hence, given a singularity $L$ of the wave trace in the interval $(0, \ell)$, we know exactly to which interval $[t_q, T_q]$ it belongs to, in other words the number of bounces $q$ of $L$ can be heard. Thus $\hat{\sigma}_{1, q}$ is a spectral invariant, proving Proposition
\ref{MMSPINV}.

\begin{rema} Proposition \ref{MMSPINVb} does not rule out that there might exist 
two distinct  $(1, q)$ orbits of $\Omega$ of the same length. It also does
not imply that the set of lengths in $\lcal_{1, q}$ is finite, nor that  the corresponding fixed point sets are clean, nor in the clean case
 that the individual terms $e_{\gamma}(t)$ in \eqref{e gamma} are spectral invariants.  
If $\Omega$ is isospectral to $E$, it says that $\hat{\sigma}_{1,q}^{\Omega}(t) = \hat{\sigma}_{1,q}^E(t)$. It is not apriori clear that the fixed point sets of the
billiard map of $\Omega$ must be clean, nor that $\mathcal L (\Omega) = \mathcal L (E)$, since cancellations may occur in the sums.  \end{rema}

		\section{Length spectrum and wave trace: Proof of Theorem
		\ref{length spec = singsupp} }
		
		For inverse spectral problems it is important to know which lengths are in the singular support of the wave trace. The following proposition was proved in \cite{MM} for any smooth strictly convex domain (satisfying a non-coincidence condition) but only for $q$ sufficiently large. For  nearly circular domains, we improve their result by showing that it holds for all $q \geq 2$. 
   \begin{prop}\label{tT in singsupp} If $\Omega$ is nearly circular in $C^8$, then for all $q \geq 2$ we have
   		$$ \{t_q(\Omega), T_q (\Omega)\} \subset \text{SingSupp}\, w_\Omega(t).$$
   \end{prop}
   
In fact, we prove the stronger Theorem
		\ref{length spec = singsupp}. 
		
	\begin{rema} For a smooth strictly convex domain $\Omega$ that is not necessarily nearly circular, if we add the assumption that 
			\begin{equation} \label{non nearly circular}\text{$|\d \Omega|$ is not a limit point from below of the set $\bigcup_{p \geq 2, q \geq 2} \mathcal L_{p, q}(\Omega)$}, \end{equation} 
			then there exists $q_0(\Omega)$ sufficiently large such that:
			\begin{equation} \label{q large poisson relation} \forall q \geq q_0(\Omega): \qquad   \mathcal L_{1, q} \subset \text{SingSupp}\, w_\Omega(t). \end{equation}
		
	Although formally speaking \cite{MM} only proves the inclusion  $\{ t_q, T_q \} \subset \text{SingSupp}\, w_\Omega(t)$, but essentially the same argument follows to show \eqref{q large poisson relation}. As shown in \cite{MM}, condition \eqref{non nearly circular} holds for a generic class of smooth strictly convex domains, hence the inclusion \eqref{q large poisson relation} holds generically. 
		
			\end{rema}

\begin{proof}  We argue by contradiction. Fix $q \geq 2$.  To prove Theorem \ref{length spec = singsupp},  assume $t_0 $ belongs to $ \mathcal{L}_{1, q}(\Omega)$ but not to the singular support of $w_\Omega(t)$. Then, there is an open interval $J_1$ near $t_0$ such that $\hat{\sigma_{1,q}}(t)$ is smooth in $J_1$. We then choose any non-negative cutoff function $\rho_q(t)$ supported in $J_1$, which is positive exactly on a proper open subinterval $J_2$ of $J_1$ containing $t_0$. In particular,  $\rho_q(t_0) >0$.  In addition we assume that the boundary points of $\supp \rho_q(t)=\bar{J_2}$ are not critical values of $L_q(s)$. This can be done because  by Sard's theorem the  set of critical values  of $L_q(s)$ has measure zero.  By our assumption on $t_0$, the inverse Fourier transform of  $\rho_q(t)\hat{\sigma_{1,q}}(t)$ must be rapidly decaying. We will see that this  leads us to a contradiction via the following theorem of Soga:
	\begin{theo}[Soga \cite{So}] \label{Soga} Consider an oscillatory integral 
		$$I(\lambda) = \int_\R e^{i \lambda \phi(x)} a(x) dx, \quad \lambda \geq 1,$$
		where $\phi(x)$ and $a(x)$ are smooth and $a(x)$ is compactly supported. Furthermore, assume that $a(x) \geq 0$ for all $x$,  and $a(x) >0$ for at least one degenerate critical point of $\phi(x)$. Then $ I(\lambda) \neq \cO(\lambda^{-\infty})$. In fact $ \lambda^{m}I(\lambda)$ is not in $L^2 (1, +\infty)$ for some $m < \frac12$.  
		\end{theo}
	
	To exploit this result, let us compute the inverse Fourier transform of $ \rho_q(t)\hat{\sigma}_{1,q}(t)$ using the parametrix \eqref{trace parametrix} as follows:
	\begin{align*}I_q(\lambda) = &  \int_\R \rho_q(t)\hat{\sigma}_{1,q}(t) e^{i \lambda t} \, dt + \cO(\lambda^{-\infty}) \\
	 =& \frac{\lambda^{3/2}}{2} \int_\R \int_0^\infty \int_{\d \Omega} e^{i \lambda (t +\xi(t - L_q(s)) + \frac{i \pi r_q}{4}} a(q, t, s, \lambda \xi) \, \xi^{\frac12} \rho_q(t)  ds d \xi dt  \\
	 & + \frac{\lambda^{3/2}}{2} \int_\R \int_0^\infty \int_{\d \Omega}  e^{i \lambda (t - \xi(t - L_q(s))- \frac{i \pi r_q}{4}}  \bar{a}(q, t, s, \lambda \xi) \, \xi^{\frac12} \rho_q(t)  ds d \xi dt  \\
	  \qquad & + \cO(\lambda^{-\infty}).\end{align*} 
	  We perform the stationary phase lemma in the $d\xi dt$ integral. We have two phase functions, namely 
	  $$ \Phi_1(t, s, \xi) = t +\xi(t - L_q(s)), \quad  \Phi_2(t, s, \xi) = t -\xi(t - L_q(s)).$$
	  Since the critical point of $\Phi_1$ is  given by $\xi =-1$ and $t =L_q(s)$, the first integral must be rapidly decaying as $\xi =-1$ is not in the domain of the integral. The critical points of $\Phi_2$ are given by $\xi =1$ and $t = L_q(s)$, therefore by the stationary phase lemma we get
	  $$I_q(\lambda) = \pi \lambda^{1/2} e^{\frac{-i\pi r_q}{4}} \int_{\d \Omega} e^{i \lambda L_q(s)}  a_0 (q, s) \rho_q(L_q(s))  ds + \cO (\lambda^{-1/2}). $$
	  We know by Lemma \ref{critical points}, that the critical points of $L_q(s)$ correspond to the $(1, q)$ periodic orbits of $\beta$. Let $s_0$ be  a critical point of $L_q(s)$ with $L_q(s_0)=t_0$. If $s_0$ is degenerate then by Theorem \ref{Soga}, we get a contradiction because $\lambda^{m-1} \in L^2(1, \infty)$ for $m <\frac12$. Now suppose $s_0$ is non-degenerate. The only remaining case is when all the periodic orbits whose  lengths are in the support of $\rho(t)$, are also non-degenerate and are finite. This is because if any such critical point is degenerate, we get a contradiction by the same argument as above. Also, if there are infinitely many such non-degenerate critical points, they must accumulate at a degenerate critical point. This accumulation point cannot be in the interior of the support of $\rho_q(t)$ or we would get a contradiction again, so it must be on its boundary. But we chose $\rho_q(t)$ so that the boundary points of its support are not critical values of $L_q(s)$.   Finally suppose the set of critical points of $L_q(s)$, whose corresponding critical values are in $\supp \rho_q(t)$, is finite and consists of non-degenerate orbits (hence each must be a local max or a local min).  We shrink the support of $\rho_q(t)$ so it contains only the critical value $t_0$. We shall use $\{s_0, \dots s_r\}$ for the set of critical points of $L_q(s)$ with critical value $t_0$. By the stationary phase lemma, we get
	  $$ I_q(\lambda) = \sqrt{2} \pi^{3/2} e^{-\frac{i\pi r_q}{4} +i\lambda t_0} \rho_q(t_0) \sum_{j=0}^r  e^{ i \frac{\pi}{4}\text{sign}( L''_q(s_j))} \frac{a_0(q, s_j)}{\sqrt{| L''_q(s_j) |}} + \cO(\lambda^{-1/2}). $$
	  Since $\text{sign}( L''_q(s_j))=1$ or $-1$, and $a_0(q, s) >0$, the sum cannot cancel to zero. 
	  
	 \end{proof}
 Before we present the proof of the main theorem, we state a key corollary of Theorem \ref{length spec = singsupp}. 
 \begin{coro} \label{big cor} For nearly circular domains in $C^8$, one has
 	\begin{equation} \label{q>=2} \text{SingSupp}\, w_\Omega(t) \cap (0, |\d \Omega| ) = \bigcup_{q\geq 2} \mathcal{L}_{1, q}(\Omega), \end{equation}
 	and \begin{equation} \label{q>=3}\text{SingSupp}\, w_\Omega(t) \cap (5, \, |\d \Omega| ) = \bigcup_{q\geq 3} \mathcal{L}_{1, q}(\Omega). \end{equation}
 \end{coro}
\begin{proof} The first statement follows from \eqref{AM}, Theorem \ref{length spec = singsupp}, and Lemma \ref{winding twice or more}. 
	To show the second statement, we note that by Lemma \ref{T-t}, we have
	$$ L_q(s) = 2q \sin(\pi/q) +  \cO(\| f\|_{C^2}).$$
	Therefore,
	$$ T_2(\Omega) = 4 + \cO(\| f\|_{C^2}), \quad t_3(\Omega) = 3 \sqrt{3} + \cO(\| f\|_{C^2}).$$
	Clearly if we choose $\cO(\| f\|_{C^2})$ sufficiently small, we have  $ T_2 (\Omega)< 5 < t_3(\Omega)$. Then, \eqref{q>=3} follows from this and \eqref{q>=2}. 
	
	\end{proof}

   \section{Proof of the main theorem}
  Suppose $\Omega$ is a smooth domain, whose $\Delta$ spectrum with respect to Dirichlet (or Neumann) boundary condition is identical with the one of an ellipse  $E_\ep$ of eccentricity $\ep < \ep_0$.  By Lemma \ref{close} and Corollary \ref{close 2}, there is a rigid motion after which $\Omega$ is $C^{n}$ for every $n \in \mathbb N$. Let us denote the wave traces of $\Omega$ and $E_\ep$ by $w_\Omega(t)$ and $w_{E_\ep}(t)$, respectively.  Since $\Omega$ and $E_\ep$ are isospectral, we must in particular have 
   $$ \text{Sing Supp}\; w_\Omega(t) \cap (0, \ell_\ep) =  \text{Sing Supp} \; w_{E_\ep}(t) \cap (0, \ell_\ep).$$
   Here $\ell_\ep$ is the length of $E_\ep$ which equals the length of $\Omega$ by the known fact that the perimeter of a domain is a spectral invariant. By Corollary \ref{big cor}, for $\ep$ sufficiently small, we must have 
  $$ \bigcup_{q \geq 3} \mathcal{L}_{1,q}(\Omega) =  \bigcup_{q \geq 3} \mathcal{L}_{1,q}(E_\ep).$$
  As we saw in \eqref{length spectrum of ellipse}, for an ellipse $E_\ep$ we have
  $$ \bigcup_{q \geq 3} \mathcal{L}_{1,q}(E_\ep) = \{ T_3(\ep) < T_4(\ep) < \cdots < T_q(\ep) < T_{q+1}(\ep) <  \cdots \}, $$
  which is a monotonically increasing sequence converging to $\ell_\ep$ whose gaps sequence $T_{q+1}(\ep) - T_q(\ep)$ is monotonically decreasing. Thus $ \bigcup_{q \geq 3} \mathcal{L}_{1,q}(\Omega)$ must be a sequence with the same properties. We claim that this implies that for all $q \geq 3$, we have $T_q(\Omega) = t_q(\Omega)$. Assume not. So for some $q \geq 3$, $t_q(\Omega) \neq T_q(\Omega)$, i.e. the the $q$-loop function $L_q(s)$ is not a constant. Then by Lemma \ref{structure of lengths}, if $q \geq q_0$, we have 
  $$ T_q(\Omega) - t_q(\Omega) \leq  \frac{1}{100(q+1)^3} <  \frac{1}{10(q+1)^3} \leq t_{q+1}(\Omega) - T_q(\Omega), $$
  and if $q<q_0$ we have
  $$ T_q (\Omega)- t_q(\Omega) \leq  \frac{\eta_0}{100} <  \frac{\eta_0}{10} \leq t_{q+1}(\Omega) - T_q(\Omega). $$ Either way, we get 
  $$ 0<T_q(\Omega) -t_q(\Omega) < t_{q+1}(\Omega) - T_q(\Omega).$$ This shows that the sequence of gaps of $ \bigcup_{q \geq 3} \mathcal{L}_{1,q}(\Omega)$  fluctuates and is not decreasing, thus a contradiction. Therefore, for all $q \geq 3$ we must have $L_q(s)$ is a constant function of $s$, or equivalently there is a smooth convex  caustic $\Gamma_q$ of rotation number $1/q$ consisting of $(1, q)$ periodic orbits. In fact the caustic $\Gamma_q$ in the phase space $\Pi$ of $\Omega$ is given by 
  $$\Gamma_q = \{ (s, \phi_q(s))\}, $$
  where $\phi_q(s)$ is the $q$-loop angle defined by Theorem \ref{loop angle}. In the language of \cite{ADK}, this precisely means that $\Omega$ is rationally integrable. The following dynamical theorem of \cite{ADK} is the final major step in our argument. 
  \begin{theo}[Avila, De Simoi, and Kaloshin]
  	Let $\Omega$ be a $C^{39}$ smooth domain that is rationally integrable and is $C^{39}$ sufficiently close to the unit disk. Then $\Omega$ is an ellipse.
   \end{theo}
Hence if we choose $n=39$, we obtain that $\Omega$ must be an ellipse. By Proposition \ref{tT in singsupp}, we know that for sufficiently small $\ep$, the lengths $t_2$ and $T_2$ are spectral invariants. For an ellipse these correspond to the bouncing ball orbits on the minor and major axes, respectively, thus $\Omega$ and $E_\ep$ must be isometric.  This concludes the proof of our main theorem.

   \subsection{Second proof of the Theorem \ref{MAINTHEO2}}
Assume that
 $\Omega$ is nearly circular in $C^8$ and  that it is isospectral to an ellipse 
$E_\ep$ of small eccentricity. We recall from Proposition \ref{MMSPINVb} that 
$$ \hat \sigma_{1,q}^\Omega(t) =  \hat \sigma_{1,q}^{E_\ep}(t). $$
We remember from \eqref{sigma} that $\hat \sigma_{1,q}(t) = \hat \chi_q(t) w(t)$, where $\hat \chi_q(t)$ is a cutoff supported near $[t_q, T_q]$ and equals one there. 
Taking Fourier transform of this equation and inserting the Marvizi-Melrose parametrix we get for $\lambda >0$
$$\int_0^{\ell_\ep} e^{ i \lambda L_q(s)}
a(q, s, \lambda) d s = \int_0^{\ell_\ep} e^{ i \lambda L_q^{E_\ep}(s)}
a^{E_\ep}(q, s, \lambda) d s +\cO(\lambda^{-\infty}), $$
where
$L_q(s)$ and $L_q^{E_\ep}(s)$ are the $q$-length functions of $\Omega$ and $E_\ep$, respectively, and $a(q, s, \lambda)$ and $a^{E_\ep}(q, s, \lambda)$ are the corresponding complete
amplitudes of the trace parametrix \eqref{trace parametrix}. They are classical symbols of order zero, i.e. polyhomogneous functions
of $\lambda$ with orders descending by unit steps. We denote their
symbol expansions as $\lambda \to \infty$ as follows: 
\begin{equation} \label{SYMBOL} \left\{ \begin{array}{l} a(q, s, \lambda) \sim \sum_{j=0}^{\infty} a_j(q, s) \lambda^{-j}, \\ \\
a^{E_{\ep}} (q, s, \lambda) \sim \sum_{j=0}^{\infty} a_j^{E_{\ep}}(q, s) \lambda^{-j}. \end{array}
\right. \end{equation}
The asymptotics are the standard ones for symbols, i.e. $a - \sum_{j \leq N}a_j \in S^{-(N-1)}. $ We comment that when $\lambda < 0$ we need to replace $a(q, s, \lambda)$ and $a^{E_{\ep}} (q, s, \lambda)$ with their complex conjugates. 
Note that $L_q^{E_\ep}(s)$ is the constant $t_q(\ep)=T_q(\ep)$. Moving the
constant phase factor to the left side gives,

\begin{lemm} \label{COR} The integral,  $$b_{1,q}(\lambda): = \int_0^{\ell_\ep} e^{ i \lambda (L_q(s)- T_q(\ep))}
a(q, s, \lambda) d s  = \sum_{j=0}^ \infty \int_0^{\ell_\ep} a^{E_\ep}_j(q, s) ds \; \lambda^{-j}  + \cO(\lambda^{-\infty})$$ is a poly-homogeneous symbol of order zero, i.e. it belongs to $S^0:=S_{1,0}^0(\mathbb R)$. \end{lemm}


\begin{coro} \label{CONORMALCOR} 

The Fourier transform \begin{equation} \label{FSIGMA} \hat{b}_{1,q}(t)  = \int_{\R} \int_{S^1} e^{- i \lambda t}  e^{ i \lambda (L_q(s)- T_q(\ep))}
a(q, s, \lambda) d s d \lambda \end{equation}  of $b_{1, q}$ is  a co-normal distribution
in $I^{1/4}(\R, \{0\})$, with 
principal symbol  $\int a_0^{E_\ep}(q, s)ds$ times $|d \xi|^{\half}$ on $T_0^*\R$.
In particular, its principal symbol is strictly positive and $\hat{b}_{1,q}(t)$ is singular at, and only at, $t = \{0\}$. \end{coro}

Let us recall the definitions: 
In the notation of \cite[Section 18.2]{HoIII} (see pages 100-101), 
$u \in I^{1/4}(\R, \{0\})$ is a conormal distribution conormal to $\{0\}$ if
$(x D_x)^k u $ belongs to the same Sobolev space as $u$  (see \cite[Definition 18.2.6]{HoIII}). By
\cite[Theorem 18.2.8]{HoIII}, $u \in I^{1/4}(\R, \{0\})$ if and only if $$u(x) = \int_{\R} e^{i\tau x} a(\tau) d \tau $$
where $a \in S^0$. Such conormal distributions are sums of homogeneous distributions
 of the form $x_+^s, (x \pm i 0)^s$ which are singular only at $x =0$. Moreover, a co-normal distribution to $\{0\}$  has a `symbol', namely a half density $a_0(\xi) |d \xi |^{\half}$ on $T^*_0 \R$, where $a_0(\xi)$ is the leading order term of
 the symbol expansion of $a(\xi)$. With these definitions in hand, we give the simple proof of Corollary \ref{CONORMALCOR}.
 \begin{proof}  Corollary \ref{CONORMALCOR} follows from Lemma \ref{COR} and the definition of $I^{1/4}(\R, \{0\}). $
The fact that $a_0 > 0$ follows from \cite[Proposition 6.11]{MM}.  \end{proof}
From  Lemma  \ref{COR} and Corollary \ref{CONORMALCOR}, we  deduce the fundamental fact allowing us to give a second
proof of the main Theorem. 
\begin{lemm} \label{ONE} $L_q(s)$ has exactly one critical value. Hence,
 $L_q(s) \equiv L_q^{E_{\ep}}$.  \end{lemm}
 
 \begin{proof} Suppose that $L_q$ is non-constant. Then it has distinct maxima and minima, and at least one of these must differ from
 $T_q$. We denote the corresponding critical value by  $t_1 \not= 0.$ With no loss of generality, we assume that $t_1$ is the minimum value.
 Let $\psi(t) \in C_0^{\infty}(\R)$ be a bump function equal to $1$ in a neighborhood of $t_1$ and equal to zero in a neighborhood of $0$. 
 Then $\hat{b}_{1,q}(t) = (1 -\psi(t)) \hat{b}_{1,q}(t)  +  \psi(t) \hat{b}_{1,q}(t) $. By the Corollary \ref{CONORMALCOR} we have $ \psi(t) \hat{b}_{1,q}(t) \in C_0^{\infty}(\R)$. We will see that this would lead to a contradiction. 
 We first take inverse Fourier transform and obtain 
 \begin{align*}  \fcal^*_{t \to \tau} \left ( \psi(t) \hat{b}_{1,q}(t) \right ) = & (\check{\psi} * b_{1, q}) (\tau) \\
 = & \int_{\R} \check{\psi}(\lambda) b_{1, q}(\tau - \lambda) d \lambda\\
=  &    \int_0^{\ell_\ep}   e^{ i \tau (L_q(s)- T_q(\ep))} \left( \int_{\R} \check{\psi}(\lambda)   e^{  - i \lambda (L_q(s)- T_q(\ep))}
a(q, s, \tau - \lambda)  d \lambda \right) ds \end{align*}
 is rapidly decaying in $\tau$.

We now claim that $ A_q(s, \tau): =  \int_{\R} \check{\psi}(\lambda)   e^{  - i \lambda (L_q(s)- T_q(\ep))}
a(q, s, \tau - \lambda)  d \lambda $ is an element of $S^0(T^*S^1). $ Indeed, using the symbol expansion \eqref{SYMBOL} 
one has, as $\tau \to \infty$, 
$$A_q(s, \tau)  \sim  \sum_{j = 0}^{\infty}  \tilde a_j(q, s)  \int_{\R} \check{\psi}(\lambda)   e^{- i \lambda (L_q(s)- T_q(\ep))} \langle \tau - \lambda \rangle^{-j} 
 d \lambda, $$
 with $\langle x \rangle := (1 + x^2)^{\half}$, and for new amplitudes $\tilde a_j$ that depend linearly on $\{a_k\}_{ k\leq j}$. Then we write this as
$$ A_q(s, \tau)  \sim \sum_{j = 0}^{\infty}  \tilde a_j(q, s) \tau^{-j}  \int_{\R} \check{\psi}(\lambda)   e^{- i \lambda (L_q(s)- T_q(\ep))} \sigma^{-j} (\tau, \lambda) d \lambda, $$ where
 $ \sigma(\lambda, \tau)  : = (|\tau|^{-2} + |\frac{\lambda}{\tau} - 1|^2)^{1/2}. $ Note that $\langle \tau - \lambda \rangle^{-1} =
  |\tau|^{-1} \sigma(\lambda, \tau)^{-1}  $ and that $\sigma(\lambda, \tau)^{-1} \to 1$ 
 uniformly on compact sets in $\lambda$ as $\tau \to \infty$, and indeed has an asymptotic expansion in $\tau$. Since $\check{\psi}
 \in \scal(\R)$, the asymptotic expansion is integrable in $d\lambda$ and some arrangement gives a symbol expansion for $A_q(s, \tau)$.
 In conclusion, 
 $$\int_0^{\ell_\ep}   e^{ i \tau (L_q(s)- T_q(\ep))}  A_q(s, \tau) ds = O(\tau^{- N}), \;\; \forall N > 0. $$
 However, the phase has a critical  point, so that the conclusion contradicts \cite[Theorem 2]{So86}. 
 This concludes the second proof of Theorem  \ref{MAINTHEO2}. For the sake of completeness, we state \cite[Theorem 2]{So86}
 in the relevant dimension one. 
 \begin{theo} Let $I(\tau) : = \int_{\R} e^{i \tau \phi(x)} \rho(x, \tau) dx, $ where $\rho(x, \tau) \sim \rho_0(x) + \rho_1(x) (i \tau)^{-1}
 	+ \cdots$ as $\tau \to \infty$. Assume that  $\rho_0 \geq 0$ with $\rho_0(x) > 0$ on the minimum set of $\phi(x)$. 
 	Then for some $m \in \R$ depending only on the dimension, $\tau^m I(\tau) \notin L^2(\R_+)$. 
 	
 \end{theo}

 \end{proof}

In the final section we provide a general result which makes our article independent of \cite{So}. 

\subsection{A refined lemma} 
In this part we present the following more general theorem which is of independent interest. It shows that many assumptions that we previously used can be relaxed. 
\begin{lemm} \label{Inverse Oscillatory Integral Problem}
	Let $\phi (s)$ and $a(s, \lambda)$ be two smooth functions on $S^1 = \R /\Z$ with $a(s, \lambda)$ satisfying:
	$$ a(s, \lambda) = a_0(s) + \cO(\lambda^{-1/2 -\epsilon}), \qquad a_0(s)\; \text{smooth and positive},$$
	for some $\epsilon >0$. Assume for $| \lambda | \geq \lambda_0>0$, we have
	$$I(\lambda)= \int_{S^1} e^{-i\lambda \phi(s)} a(s, \lambda) ds = c_0 + \cO(\lambda^{-\frac12 - \epsilon}), $$
	for a constant $c_0$. 
	Then $\phi(s) \equiv 0$ on $S^1$. 
\end{lemm}
\begin{proof}
We follow the proof of \cite{So} closely. First let us discard the remainder term in $a(s, \lambda)$ and call the resulting integral $J(\lambda)$, i.e.
$$J(\lambda) = \int_{S^1} e^{-i\lambda \phi(s)} a_0(s) ds .$$
Obviously we have $J(\lambda) =c_0 +\cO(\lambda^{-1/2-\epsilon})$. 
Now, let $\psi(t)$ be a cutoff function in $\R$ that equals one on an open set containing the range of $\phi$. Let $H(t)$ be the Heaviside function at zero and define:
$$ g_0(t) = \psi(t) \int_\R H(t - \phi(s)) a_0(s) ds = \psi(t) \int_{\phi(s) \leq t} a_0(s) ds,$$
$$ g_1(t) = \psi'(t) \int_\R H(t - \phi(s)) a_0(s) ds = \psi'(t) \int_{\phi(s) \leq t} a_0(s) ds.$$
The function $g_0$ is in $H^{m}$ for every $m < 1/2$, and because it is a compactly supported distribution its Fourier transform $\widehat{g_0}(\lambda)$ is an analytic function. Also since $\int_\R H(t - \phi(s)) a(s) ds$ is smooth at the regular values $t$ of $\phi(s)$, and since $\psi'(t)$ vanishes on the range of $\phi(s)$, $g_1(t)$ is smooth, hence $\widehat{g_1}(\lambda)$ is rapidly decaying. 
We note that by a simple integration by parts, $I(\lambda)$ can be written as
\begin{align*}
 J(\lambda) = & i \lambda \int_\R e^{-i\lambda t} g_0(t) dt - \int_\R e^{-i\lambda t} g_1(t) dt \\
                   = & i \lambda \widehat{g_0}(\lambda) - \widehat{g_1}(\lambda).
 \end{align*}
 Thus by our assumption on $I(\lambda)$ (so the same for $J(\lambda)$), we obtain
 $$ \widehat{g_0}(\lambda) = -\frac{i J(\lambda)}{\lambda} + \cO(\lambda^{-\infty}) = -\frac{ic_0}{\lambda} + \cO(\lambda^{-3/2-\epsilon}), \qquad |\lambda| \geq \lambda_0. $$
 By taking inverse Fourier transform and using the Sobolev Embedding Theorem,
 we get 
 \begin{equation} \label{g_0 decomposition} g_0(t) = c_0H(t) + f(t), \end{equation}
 where $H(t)$ is the Heaviside function at $t=0$  (conormal distribution) with a jump discontinuity at $t=0$, and $f(t)$ is continuous at every $t$ and in fact belongs to the H\"older class $C^{0, \alpha}$ for every $0<\alpha <1/2+\epsilon$.   However, we will show that every critical value $t_0$ of $\phi(s)$ is a `big singularity' of $g_0(t)$. By this we mean that $g_0(t)$ is not H\"older continuous $C^{0, \alpha}$ at $t_0$ for any $\alpha > 1 /2 $. This together with the decomposition \eqref{g_0 decomposition} would imply that the only critical value of $\phi(s)$ is $0$. Since $\phi$ is a function on $S^1$, it must be zero everywhere. 
 
 So assume $t_0 = \phi(s_0)$ is a critical value of $\phi$ and $s_0$ is a critical point in its inverse image. We denote for each $h>0$, 
 $$ A_{h} = \{ s \in S^1; \; -h < \phi(s) -t_0 \leq h \}.$$
 We recall that $\psi(t) =1$ in an open set containing the image of $\phi$,  so we can choose $h$ small enough so that $\psi(t_0-h) =\psi(t_0+h) =1$. By the definition of $g_0$, we have
 $$ | g_0(t_0+h) - g_0(t_0-h) | =\int_{ A_h} a_0(s) ds.$$
 However, since $s_0$ is a critical point of $\phi$, we have $| \phi(s) - t_0| \leq c|s-s_0|^2$ for some $c>0$, and thus we have the inclusion 
 $$ \{ s \in S^1; \; c|s-s_0|^2 < h \} \subset A_h.$$ 
 We then write, 
 $$ |g_0(t_0+h) - g_0(t_0-h)|  \geq \int_{ c|s-s_0|^2 < h} a_0(s) ds \geq \sqrt{h/c} \,  \min{a_0(s)} , $$
 which by letting $h \to 0$ implies that $g_0$ is not H\"older continuous $C^{0, \alpha}$ at $t_0$ for $\alpha > 1/2$.

\end{proof}

  \subsection*{Acknowledgement} The authors are grateful to the anonymous referee for their helpful comments and suggestions. The research of H.H. is supported by the Simons Collaborations Grants for Mathematicians 638398. The research of S.Z. is partially supported by NSF grant DMS-1810747. .

\end{document}